\documentclass[10pt]{amsart}
\usepackage{amsfonts}
\usepackage{indentfirst}
\usepackage{graphicx}
\usepackage{amssymb}
\usepackage{color,xcolor}
\usepackage{graphicx,epstopdf}
\usepackage{epsfig}
\usepackage{subfigure}
\usepackage{caption}
\usepackage{mathrsfs}
\usepackage{bbm}
\usepackage{chngpage}
\usepackage{cite}
\usepackage{multirow}
\usepackage{multicol}
\usepackage{dsfont}
\usepackage{bm}
\usepackage{amssymb,amsmath,amsthm,mathrsfs}%
\usepackage{exscale}
\usepackage{amsmath}
\usepackage{relsize}

\allowdisplaybreaks

\newtheorem{theorem}{Theorem}[section]
\newtheorem{definition}[theorem]{Definition}

\theoremstyle{definition}

\newtheorem{example}[theorem]{Example}
\theoremstyle{remark}
\newtheorem{remark}[theorem]{Remark}
\newtheorem{lemma}{Lemma}[section]

\numberwithin{equation}{section}%
\numberwithin{table}{section}%
\numberwithin{figure}{section}

\def\3bar{{|\hspace{-.02in}|\hspace{-.02in}|}}
\def\td{\text{div}}
\def\tc{\text{curl}}
\def\d{\text{d}}

\begin{document}
\title[An $H^2$(curl)-conforming FEs in 2D and its applications]{An $\bm H^2$(curl)-conforming finite element in 2D and its applications  to the quad-curl  problem}
\author{Qian Zhang}
\email{qianzhang@csrc.ac.cn}
\address{Beijing Computational Science Research Center, Beijing, China}

\author{Lixiu Wang}
\email{lxwang@csrc.ac.cn}
\address{Beijing Computational Science Research Center, Beijing, China}

\author{Zhimin Zhang}
\email{zmzhang@csrc.ac.cn; zzhang@math.wayne.edu}
\address{Beijing Computational Science Research Center, Beijing, China; Department of Mathematics, Wayne State University, Detroit, MI 48202, USA. }

\begin{abstract}
In this paper, we first construct the $H^2$(curl)-conforming finite elements both on a rectangle and a triangle.  They possess some fascinating properties which have been proven by a rigorous theoretical analysis. Then we apply the elements to construct a finite element space for discretizing quad-curl problems.  Convergence orders $O(h^k)$ in the $H$(curl) norm and $O(h^{k-1})$ in the $H^2$(curl) norm are  established. Numerical experiments are provided to confirm our theoretical results.
\end{abstract}	
\maketitle

\section{Introduction}
The previous academic works \cite{Monk2003,Jin1993Finite,Sun2016Finite,Li2013Time,WGmaxwell} present that  the Sobolev space $H(\tc;\Omega)$ plays a vital role in the variational theory of Maxwell's equations. Hence the so-called $H$(curl)-conforming finite elements (FEs) in this space are needed for approximating the Maxwell's equations and related problems. A good choice was proposed by N\'ed\'elec in  \cite{Lec1980A,Lec1986A}  using  quantities (moments of tangential components of vector fields) on edges and faces, thus the name edge element. Edge elements can be applied to handle complex geometry and discontinuous electromagnetic properties which occur when different materials are involved in the electromagnetic field. However, as the problems we are concerned with(e.g. transmission problems \cite{Cakoni2017A} and magnetohydrodynamics equations  \cite{Qingguo2012A}) get increasingly complex, functions in $H^2(\tc)$ are required, thus the N\'ed\'elec elements fail to meet our continuity requirements. Moreover, in \cite{Qingguo2012A, Sun2016A}, the authors  believe constructing such an element should be very difficult and expensive. These reasons motivate us to research FEs conforming in $H^2$(curl).

This paper starts by describing a class of curl-curl-conforming or $H^2$(curl)-conforming FEs which, to the best of the authors' knowledge, are not available in the existing literature. In our construction, the number of the degrees of freedom (DOFs) for the lowest-order element is only 24 for both a triangle and a rectangle. Recall the number of the DOFs for the $H^2$-conforming element, it is $18\times 2 =  36$ for a triangle (Bell element) and $16\times 2 = 32$ for a rectangle (BFS element). Also, it may lead to false solutions when using the $H^2$-conforming element to solve Maxwell's problems. The unisolvence and conformity can be verified by a straightforward mathematical analysis. Moreover, our new elements possess some good properties by which we obtain the error estimate of interpolation.

In the second part, we use our elements to solve the quad-curl equations which are involved in  various practical problems, such as in inverse electromagnetic scattering theory \cite{Cakoni2017A,Monk2012Finite, Sun2016A} or in magnetohydrodynamics \cite{Zheng2011A}.
Unlike the low-order electromagnetic problem, which have been extensively studied both in mathematical theory  and  numerical methods \cite{Monk2003,Peter1992A,Teixeira2008Time,Jin1993Finite,Daveau2009A,Li2013Time,Monk1994Superconvergence}, only limited work has been done for quad-curl problems.
Zheng et al. developed a nonconforming finite element method for the problem in \cite{Zheng2011A}. Although the method has small number of DOFs, it bears the disadvantage of low accuracy. Based on N\'ed\'elec elements, a discontinuous Galerkin method and a weak Galerkin method were presented in \cite{Qingguo2012A} and \cite{quadcurlWG}, respectively. Another approach  to deal with the quad-curl operator is to introduce an auxiliary variable and reduce the original problem to a second-order system \cite{Sun2016A}.  Very recently, Brenner et al. proposed a Hodge decomposition method in \cite{Brenner2017Hodge} by using  Lagrange elements.
Zhang used a different mixed scheme \cite{Zhang2018M2NA}, which relaxes the regularity require-
ment of the solution.
Different from all aforementioned methods, in this paper, we propose a conforming FE method by using our $H^2$(curl) elements and give the error estimate in the sense of $H^2(\tc)$-, $H(\tc)$- and $L^2$- norms.

The rest of the paper is organized as follows. In Section 2 we list some function spaces and notations. Section 3 is the most technical part, where we define the $H^2$(curl)-conforming FEs on rectangles and estimate the interpolation error.  In Section 4 we give the definition of FEs and some properties on triangles. In Section 5 we use our FEs to solve the quad-curl problems and give the error estimates.  In section 6 we provide numerical examples to  verify correctness and efficiency of our method. Finally,  some concluding remarks and possible future works are given in  Section 7. In the Appendix we list the basis functions on a reference rectangle and triangle of the lowest-order element.

\section{Preliminaries}
Let $\Omega\in\mathbb{R}^2$ be a convex Lipschitz domain. $\bm n$ is the unit outward normal vector to $\partial \Omega$. We adopt standard notations for Sobolev spaces such as $H^m(D)$ or $H_0^m(D)$ on a sub-domain $D\subset\Omega$ equipped with the norm $\left\|\cdot\right\|_{m,D}$ and the semi-norm $\left|\cdot\right|_{m,D}$. If $m=0$,  the space $H^0 (D)$ coincides with $ L^2(D)$ equipped with the norm $\|\cdot\|_{D}$, and when $D=\Omega$, we drop the subscript $D$. We use  $\bm H^m(D)$   and ${\bm L}^2(D)$ to denote the vector-valued Sobolev spaces $\left(H^m(D)\right)^2$ and $\left(L^2(D)\right)^2$.

We define
\begin{align*}
H(\text{curl};D)&:=\{\bm u \in {\bm L}^2(D):\; \nabla \times \bm u \in L^2(D)\},\\
H^2(\text{curl};D)&:=\{\bm u \in {\bm L}^2(D):\; \nabla \times \bm u \in L^2(D),\;(\nabla \times)^2 \bm u \in \bm L^2(D)\},
\end{align*}
whose scalar products and norms are defined by
$(\bm u,\bm v)_{H^s(\tc;D)}=(\bm u,\bm v)+\sum_{j=1}^s(\nabla\times \bm u,\nabla\times \bm v)$ and
$\left\|\bm u\right\|_{H^s(\tc;D)}=\sqrt{(\bm u,\bm u)_{H^s(\tc;D)}}$ with $s=1,\;2.$
The spaces $H_0^s(\text{curl};D)\;(s=1,\;2)$ are defined as follows:
\begin{align*}
&H_0(\text{curl};D):=\{\bm u \in H(\text{curl};D):\;{\mathbf n}\times\bm u=0\; \text{on}\ \partial D\},\\
&H^2_0(\text{curl};D):=\{\bm u \in H^2(\text{curl};D):\;{\mathbf n}\times\bm u=0\; \text{and}\; \nabla\times \bm u=0\;\; \text{on}\ \partial D\}.
\end{align*}
The space of $\bm L^2(D)$ functions with square-integrable divergence is denoted by $H(\td;D)$ and defined by
\[H(\text{div};D) :=\{\bm u\in {\bm L}^2(D):\; \nabla\cdot \bm u\in L^2(D)\},\]
with the associated inner product
$(\bm u,\bm v)_{H(\td;D)}=(\bm u,\bm v)+(\nabla\cdot \bm u,\nabla\cdot \bm v)$ and the norm
$\left\|\bm u\right\|_{H(\td;D)}=\sqrt{(\bm u,\bm u)_{H(\td;D)}}$.

Taking the divergence-free condition into account, we define
\begin{align}\label{spac-01}
&Y=\left\{\bm u\in H_0(\tc;D)\big|(\bm u,\nabla p)=0,\;\;\forall p\in H_0^1(D)\right\},\\
&X=\left\{\bm u\in H_0^2(\tc;D)\big|(\bm u,\nabla p)=0,\;\;\forall p\in H_0^1(D)\right\},\label{spac-02}\\
&H(\text{div}^0;D) :=\{\bm u\in H(\text{div};D):\; \nabla\cdot \bm u=0\;\text{in}\; D\}.
\end{align}
$Q_{i,j}(D)$ represents the space of polynomials on $D$ with the  degrees of the variables $x,\;y$ no more than $i,\;j$. For simplicity, we drop a  subscript $i$ when $i=j$. And $P_i(D)$ represents the space of polynomials on $D$ with a degree of no more than $i$.
\begin{lemma}\label{Helm}
	Let $\nabla H_0^1(\Omega)$ be the set of gradients of functions in $H_0^1(\Omega)$. Then $\nabla H_0^1(\Omega)$ is a closed subspace of $H_0^2(\tc;\Omega)$ such that
	\[H_0^2(\tc;\Omega)=X\oplus \nabla H_0^1(\Omega),\]
\end{lemma}
\begin{proof}
	Using the fact that $\nabla H_0^1(\Omega)$ is a closed subspace of $H_0(\tc;\Omega)$, it's trivial to get the following three conditions:
	\begin{itemize}
		\item Due to $\nabla H_0^1(\Omega)\subset H_0(\tc;\Omega)$ and $\nabla\times\nabla\times\nabla H_0^1(\Omega)=\{0\}$, $\nabla H_0^1(\Omega)\subset H_0^2(\tc;\Omega)$.
		\item if $\bm x,\;\bm y\in \nabla H_0^1(\Omega)$, then $\alpha\bm x+\beta \bm y\in  \nabla H_0^1 (\Omega)$, $\forall\alpha,\beta\in \mathbb R$.
		\item if $\{\bm x_m\}_{m=1}^{\infty}\subset \nabla H_0^1(\Omega)$, and
		$\left\|\bm x_m-\bm x\right\|_{H^2(\tc;\Omega)}\rightarrow 0$, then $\left\|\bm x_m-\bm x\right\|_{H(\tc;\Omega)}\rightarrow 0$, thus $x\in \nabla H_0^1(\Omega)$.
	\end{itemize}
	Consequently, $\nabla H_0^1(\Omega)$ is a closed subspace of $H_0^2(\tc;\Omega)$, which completes the proof.
\end{proof}

\section{$H^2$(\text{curl})-conforming elements on rectangles}\label{rectelem}
The elements we shall define and analyze in this section will be used to discretize the electric field in some high-order Maxwell's equations. We first introduce the following lemma. It shows that  a finite element is conforming in $H^2(\tc)$ if $\nabla\times\bm u$ and  the tangential component of $\bm u$are continuous across the edges of elements.
\begin{lemma}\label{prob2}
Let $\text{K}_1$ and $\text{K}_2$ be two non-overlapping Lipschitz domains having a common edge $\Lambda$ such that $\overline{\text{K}_1}\cap\overline{\text{K}_2} = \Lambda$. Assume that ${\bm u}_1 \in H^2(\text{curl};\text{K}_1)$,
${\bm u}_2 \in H^2(\text{curl};\text{K}_2)$, and $\bm u \in \bm L^2(\text{K}_1 \cup \text{K}_2 \cup \Lambda)$ is defined by
\begin{equation*}
\displaystyle{\bm u}=
\begin{cases}
&\bm u_1,\quad  \text{in}\ \text{K}_1,\\[0.2cm]
&\bm u_2,\quad  \text{in}\ \text{K}_2.
\end{cases}
\end{equation*}
Then $\bm u_1 \times \bm n_1 = -\bm u_2 \times \bm n_2$\ and $\nabla\times \bm u_1=\nabla\times \bm u_2$ on $\Lambda$ implies that $\bm u \in H^2(\text{curl};\text{K}_1 \cup \text{K}_2 \cup \Lambda)$, where $\bm n_i$ ($i=1,2$) is the unit outward normal vector to $\partial K_i$, and note that  $\bm n_1 = - \bm n_2$.
\end{lemma}
\begin{proof}
For any function $\bm \phi \in \big(C_0^\infty(\text{K}_1 \cup \text{K}_2 \cup \Lambda)\big)^2$,
\begin{equation*}
\begin{split}
&\int_{\text{K}_1 \cup \text{K}_2 \cup \Lambda} \bm u\cdot (\nabla\times)^2 \bm \phi \d \bm x\\
= &\int_{\text{K}_1} \nabla\times\bm u_1 \nabla\times \bm \phi \d \bm x+ \int_{\text{K}_2} \nabla\times\bm u_2 \nabla\times \bm \phi \d \bm x + \int_\Lambda ( \bm u_1 \times \bm n_1+ \bm u_2 \times \bm n_2)\nabla\times\bm \phi\d s\\
=&\int_{\text{K}_1} \nabla\times\nabla\times\bm u_1\cdot \bm \phi \d\bm x + \int_{\text{K}_2} \nabla\times\nabla\times\bm u_2\cdot \bm \phi \d \bm x -\int_\Lambda(\bm \phi\times\bm n_1)\nabla\times\bm u_1+(\bm \phi \times\bm n_2 )\nabla\times\bm u_2\d s \\
+& \int_\Lambda ( \bm u_1 \times \bm n_1+ \bm u_2 \times \bm n_2)\nabla\times\bm \phi\d s,
\end{split}
\end{equation*}
where $\bm u_i = \bm u|_{K_i}, \;i = 1,\;2.$
The assumption $\bm u_1 \times \bm n_1 = -\bm u_2 \times \bm n_2$\ and $\nabla\times \bm u_1=\nabla\times \bm u_2$ on $\Lambda$ leads to
\begin{align*}
&\int_{\text{K}_1 \cup \text{K}_2 \cup \Lambda} \bm u\cdot (\nabla\times)^2 \bm \phi \d \bm x=\int_{\text{K}_1 \cup \text{K}_2 \cup \Lambda} w\nabla\times \bm \phi \d \bm x,\\
&\int_{\text{K}_1 \cup \text{K}_2 \cup \Lambda} \bm u\cdot (\nabla\times)^2 \bm \phi \d \bm x=\int_{\text{K}_1 \cup \text{K}_2 \cup \Lambda} \bm v\cdot\bm \phi \d \bm x,
\end{align*}
where $w|_{K_i}=\nabla\times( \bm u|_{K_i})$ and $\bm v|_{K_i}=(\nabla\times)^2 ( \bm u|_{K_i})$, $i=1,\;2$ are  the weak curl and the weak curl-curl of $\bm u$.
Thus, we complete the proof.
\end{proof}
\subsection{$H^2$(\text{curl})-conforming elements }
\begin{definition}\label{def} For any integer $k \geq 3$, the ${H}^2(\mathrm{curl})$-conforming element is defined by the triple:
\begin{equation*}
\begin{split}
&\hat{K}=(-1,1)^2,\\
&P_{\hat{K}} = Q_{k-1,k}\times Q_{k,k-1},\\
&\Sigma_{\hat{K}} = \bm M_{\hat{p}}(\hat {\bm u}) \cup \bm M_{\hat{e}}(\hat {\bm u}) \cup \bm M_{\hat{K}}(\hat {\bm u}),
\end{split}
\end{equation*}
where $\bm M_{\hat{p}}(\hat {\bm u})$, $\bm M_{\hat{e}}(\hat {\bm u})$ and $\bm M_{\hat{K}}(\hat {\bm u})$ are the DOFs defined as follows:\\
$\bm M_{\hat{p}}(\hat {\bm u})$ is the set of DOFs given on all vertex nodes and edge nodes $\hat{p}_{i}$:
\begin{equation}\label{def2}
\bm M_{\hat{p}}(\hat {\bm u})=\left\{(\hat{\nabla}\times \hat{\bm u})(\hat p_{i}),\; i=1,\;2,\cdots,4(k-1) \right\},
\end{equation}
where we choose the points $\hat p_{i}$ at : $4$ vertex nodes and $(k-2)$ distinct nodes on each edge.\\
$\bm M_{\hat{e}}(\hat {\bm u})$ is the set of DOFs given on all edges $\hat{e}_i$ of $\hat{K}$, each with the unit  tangential vector $\hat{\bm \tau}_i$:
\begin{equation}\label{def3}
\bm M_{\hat{e}}(\hat {\bm u})= \left\{\displaystyle{\int}_{\hat{e}_i}\hat {\bm u}\cdot \hat{\bm \tau}_i \hat{q}\mathrm d\hat{s},\ \forall \hat{q}\in P_{k-1}(\hat{e}_i), i=1,2,3,4 \right\}.
\end{equation}
 $M_{\hat{K}}(\hat {\bm u})$ is the set of DOFs given in the element $\hat{K}$:
\begin{align}\label{def5}
&\bm M_{\hat{K}}(\hat {\bm u})=\left\{\int_{\hat{K}} \hat{\bm u}\cdot \hat{\bm q}\mathrm d \hat V,\ \forall \hat{\bm q}\in  Q_{k-2}(\hat K)\cdot \hat{\bm x}\ \text{and}\ \hat {\bm q}=\hat{\nabla}\times\hat{\varphi},\ \forall \hat{\varphi}\in  \tilde{Q}_{k-3}(\hat K) \right\},
\end{align}
where $\hat{\bm x}=(\hat x_1,\;\hat x_2)^T$ and $\tilde{Q}_{k-3}$ represents the space of polynomials in $Q_{k-3}$ without constant term 0.
\end{definition}
Now we have $4(k-1)$ node DOFs, $4k$ edge DOFs and $(k-1)^2+(k-2)^2-1$ element DOFs. After calculation, we can obtain
\[\dim(P_{\hat{K}})= 4(k-1) + 4k + (k-1)^2+(k-2)^2-1 = 2k(k+1).\]
Since $k\geq 3$, the minimum number of DOFs  is 24.
\begin{lemma}\label{wel-def-01}
	The DOFs \eqref{def2}-\eqref{def5} are well-defined for any $\hat{\bm u}\in \bm H^{1/2+\delta}(\hat{K})$ and $\hat{\nabla}\times\hat{\bm u}\in H^{1+\delta}(\hat{K})$, with $\delta>0$.
\end{lemma}
\begin{proof}
	By the embedding theorem, we have $\hat{\nabla}\times\hat{ \bm u} \in H^{1+\delta}(\hat{K})\subset C^{0,\delta}(\hat{K}),$ then the DOFs in $M_{\hat{p}}$ are well-defined.
	It follows Cauchy-Schwarz inequality  that the DOFs defined in $M_{\hat e} (\hat{\bm u})$ and $M_{\hat{K}}(\hat{\bm u})$ are well-defined since $\hat{\bm u} \in \bm H^{1/2+\delta}(\hat{K})$ and $\hat{\bm u}|_{\partial \hat K} \in \bm H^{\delta}(\partial \hat K)$.
\end{proof}
\begin{theorem}
The finite element given by Definition \ref{def} is unisolvent and conforming in $H^2(\mathrm{curl})$.
\end{theorem}
\begin{proof}
(\romannumeral 1). To prove the $H^2$(curl) conformity, it suffices to prove $\hat{\bm u} \times \hat{\bm n}_i = 0$ and $\hat{\nabla}\times \hat{\bm u}=0$ on an edge (e.g. $\hat{x}_1=-1$)  when all DOFs associated this edge vanish.  Without loss of generality, we consider the edge $\hat{x}_1=-1$. On this edge, $$\hat{\bm u}\cdot \hat{\bm \tau}_i = \hat{u}_2 \in P_{k-1}(\hat x_2).$$
By choosing $q=\hat{\bm u}\cdot \hat{\bm \tau}_i$ in \eqref{def3}, we obtain $\hat{\bm u}\cdot \hat{\bm \tau}_i=\hat{\bm u}\times \hat{\bm n}_i=0$.
Furthermore, we get $\hat{\nabla}\times \hat{\bm u}=0$ by using the $k$ vanishing DOFs defined in \eqref{def2}.

(\romannumeral 2). Now, we consider the unisolvence. It is clear that the total number of DOFs is $2k(k+1)=\dim P_{\hat K}$.
We only need to prove that vanishing all DOFs for $\hat{\bm u }\in P_{\hat{K}}$ yields $\hat{\bm u}=0$. By virtue of the fact that $\hat{\nabla}\times \hat{\bm u} = 0$ on $\partial \hat K$, we can rewritten $\hat{\nabla}\times \hat{\bm u}$ as:
\[\hat{\nabla}\times \hat{\bm u} = (1-\hat{x}_1)(1+\hat{x}_1)(1-\hat{x}_2)(1+\hat{x}_2)\hat{v},\ \hat{v}\in Q_{k-3}(\hat K).\]
Then, by using the integration by parts, we can get
\begin{equation*}
\int_{\hat{K}}\hat{\nabla}\times \hat{\bm u}\hat{ v} \d\hat{V}=\int_{\hat{K}}\hat{\bm u}\hat{\nabla}\times \hat{ v} \d\hat{V}+\int_{\partial \hat{K}}\hat{ v }\hat{\bm n}\times\hat{\bm u}\d s.
\end{equation*}
Due to \eqref{def3} and \eqref{def5}, we arrive at $\hat{\nabla}\times \hat{\bm u} = 0\ \text{in}\ \hat{K}$.
Thus, there exists a function $\hat{\phi} \in Q_{k}(\hat K)$  such that $\hat{\bm u}=\nabla\hat{\phi}$. According to (\romannumeral 1), we have $\nabla\hat{\phi}\cdot\hat{\bm {\tau}}=\nabla\hat{\phi}\times\hat{\bm n} =0$ on $\partial \hat K$, which implies that $\hat{\phi}$ can be chosen as
\[\hat{\phi}= (1-\hat{x}_1)(1+\hat{x}_1)(1-\hat{x}_2)(1+\hat{x}_2)\hat{\varphi},\ \ \hat{\varphi} \in Q_{k-2}(\hat K).\]
By applying the integration by parts again, we have
 \begin{equation*}
\int_{\hat{K}} \hat{\bm u}\hat{\bm q} \d\hat{V}=\int_{\hat{K}}\nabla\hat{\phi}\hat{\bm q} \d\hat{V}=\int_{\hat{K}}\hat{\phi}\nabla\cdot\hat{\bm q} \d \hat{V},\ \ \forall \hat{ \bm q}\in Q_{k-2}\cdot\hat{\bm x}.
\end{equation*}
Choosing $\nabla\cdot\hat{\bm q} = \hat{\varphi}$ and then using \eqref{def5}, we obtain $\hat{\varphi}=0$, i.e. $\hat{\bm u}=0$, which completes the proof.
\end{proof}

We need to extend the curl-curl-conforming element to a general parallelogram $K$. This is done by relating the finite element function on $K$ to a function on the $\hat K$. Since the elements are vectorial and we wish to relate the curl of $\bm u$ in an easy way to the curl of $\hat{ \bm u}$, we adopt the following transformation:
\begin{align}
\bm u \circ F_K = B_K^{-T} \hat{\bm u},\label{trans-1-1}
\end{align}
where the affine mapping
\begin{align}\label{trans-K}
F_K(\bm x)= B_K\hat{\bm x}+ \bm b_K.
\end{align}
By a simple computation, we have
 \begin{align}
 &(\nabla\times\bm u) \circ F_K = \frac{1}{\det(B_K)} \hat{\nabla}\times\hat{ \bm u},\label{trans-1-2}\\[0.2cm]
 &(\nabla\times\nabla\times\bm u) \circ F_K =\frac{B_K}{(\det(B_K))^2}\hat{\nabla}\times\hat{\nabla}\times\hat{ \bm u}.\label{trans-1-3}
 \end{align}
The unit tangential vector $\bm \tau$ along the edge $e$ of $K$ is achieved by the transformation \cite{Li2013Time}
 \begin{equation}\label{trans-3}
\bm \tau \circ F_K = \frac{B_K \hat{\bm \tau}}{|B_K \hat{\bm \tau}|}.
\end{equation}

Next we need to relate the DOFs on $K$ and $\hat K$ and show that they are invariant under the transformation \eqref{trans-1-1}.
\begin{lemma}\label{lemm-2}
Suppose that the function $\bm u$ and the unit tangential vector $\bm \tau$ to $\partial K$ are defined by the transformations \eqref{trans-1-1} and \eqref{trans-3}. Suppose also that the DOFs of a function $\bm u$ on $K$ are given by
\begin{align}
\bm M_{{p}}({\bm u})&=\Big\{\det(B_K)({\nabla}\times {\bm u})(p_{i}),  i=1,\;2,\;\cdots,\;4(k-1) \Big\},\label{Ldef2}\\
\bm M_{{e}}( {\bm u})&= \left\{\int_{e_i} {\bm u}\cdot {\bm \tau}_i {q}\d{s},\ \forall {q}\in P_{k-1}({e}_i), \;i=1,2,3,4 \right\}\label{Ldef3},\\
\bm M_{{K}}(\bm { u})&=\left\{\int_{{K}} {\bm u}\cdot{\bm q}\d V,\ \forall \bm q \circ F_K = \frac{B_K\hat{\bm q}}{\det(B_K)},\;\hat{\bm q}\in  Q_{k-2}(\hat K)\cdot \hat{\bm x}\right.\nonumber\\
&\quad\quad\quad\quad\left.\text{and} \ {\bm q}=\nabla\times{\varphi},\ \forall {\varphi}=\hat{\varphi}\circ F_K^{-1},\;\hat{\varphi}\in  \tilde{Q}_{k-3}( \hat K) \right\}.\label{Ldef5}
\end{align}
Then the DOFs for $\hat{\bm u}$ on $\hat K$ and for $\bm u$ on $K$ are identical.
\end{lemma}
\begin{remark}
Note that the DOFs in  $\bm M_{{p}}({\bm u})$  involve $\det(B_k)$ and  $\det(B_k)$ varies from element to element when the mesh is nonuniform.  In the computation, we need to transfer $\det(B_k)$ from  $\bm M_{{p}}({\bm u})$ to the basis functions.
\end{remark}
\subsection{The $H^2$(curl) interpolation and its error estimates}

Provided $\bm u \in \bm H^{1/2+\delta}(K)$, and $ \nabla \times \bm u \in H^{1+\delta}(K)$ with $\delta >0$ (see Lemma \ref{wel-def-01}),  we can
define an $H^2$(curl) interpolation operator on $K$ denoted as $\Pi_K$ by
 \begin{eqnarray}\label{def-inte-01}
M_p(\bm u-\Pi_K\bm u)=0,\ M_e(\bm u-\Pi_K\bm u)=0\ \text{and}\ M_K(\bm u-\Pi_K\bm u)=0,
\end{eqnarray}
where $M_p,\ M_e$ and $M_K$ are the sets of DOFs in \eqref{Ldef2}-\eqref{Ldef5}.

Before we prove the error estimates of the interpolation, we need to show some important lemmas.
Firstly, we shall show the relationship between the interpolation on a general element $K$ and the interpolation on the reference element $\hat{K}$.

\begin{lemma}\label{hatpi}
	Assume that $\Pi_K\bm u$ is well-defined. Then under the transformation \eqref{trans-1-1},
	we have
	$$\widehat{\Pi_K \bm u} = \Pi_{\hat{K}} \hat{\bm u}.$$
\end{lemma}
\begin{proof}
	Because of Lemma \ref{lemm-2} and the definition of the interpolation \eqref{def-inte-01}, we have
	\begin{eqnarray*}\label{def-inte-02}
		M_{\hat{p}}(\hat{\bm u}-\widehat{\Pi_K\bm u})=0,\ M_{\hat{e}}(\hat{\bm u}-\widehat{\Pi_K\bm u})=0\ \text{and}\ M_{\hat{K}}(\hat{\bm u}-\widehat{\Pi_K\bm u})=0.
	\end{eqnarray*}
	Then based on the unisolvence  of the DOFs, we obtain
	\begin{eqnarray}\label{def-inte-03}
	\Pi_{\hat{K}}(\hat{\bm u}-\widehat{\Pi_K\bm u})=0.
	\end{eqnarray}
	Furthermore, we have $\Pi_{\hat{K}}(\widehat{\Pi_K\bm u})=\widehat{\Pi_K\bm u}$, which, together with the equation \eqref{def-inte-03}, leads to the conclusion.
\end{proof}
\begin{lemma}\label{lem2.5}
Provided that $\bm u$ is sufficiently smooth,  we have
	\begin{align}\label{pi_I}
	\left\|\nabla\times\Pi_K\bm u-I_K\nabla \times\bm u\right\|\leq \left\|\nabla\times\bm u-I_K\nabla\times\bm u\right\|,
	\end{align}
	where $I_K$ is the Lagrange interpolation operator from $H^{1+\delta}(K)$ to $Q_{k-1}(K)$ using the same vertex and edge nodes as the $H^2(\tc)$-conforming elements.
\end{lemma}
\begin{proof}
	Without loss of generality, we only prove \eqref{pi_I} for the reference element $ \hat K$. And for the sake of brevity, we shall drop the hat notation. Applying the definition of ${\Pi}_{ K}$, we can get the value of $\nabla\times\Pi_K\bm u$ at the vertex and edge nodes $p_i\;(i=1,2,\cdots,n_p)$ as
	\[(\nabla\times\Pi_K\bm u)(p_i)=(\nabla\times\bm u)(p_i).\]
	Using the definition of Lagrange interpolation, we obtain
	\[(I_K\nabla\times\bm u)(p_i)=(\nabla\times\bm u)(p_i).\]
	It follows that  $\nabla\times\Pi_K\bm u-I_K\nabla\times\bm u=0$ on $\partial K$. Thus,
	\[\nabla\times\Pi_K\bm u-I_K\nabla\times\bm u=(1-x_1)(x_1+1)(1-x_2)(x_2+1)v,\quad v\in Q_{k-3},\]
	which, together with the  triangle inequality, leads to
    \begin{align*}
	& \left\|\nabla\times\Pi_K\bm u-I_K\nabla\times\bm u\right\|_K^2\\
	=&\int_K(1-x)(x+1)(1-y)(y+1)(\nabla\times\Pi_K\bm u-I_K\nabla\times\bm u)v\d V\\
	\leq &C\int_K(\nabla\times\Pi_K\bm u-\nabla\times\bm u)v\d V+\int_K(\nabla\times\bm u-I_K\nabla\times\bm u)(\nabla\times\Pi_K\bm u-I_K\nabla\times\bm u)\d V.
	\end{align*}
	Using the integration by parts, we have, in light of \eqref{Ldef3} and \eqref{Ldef5},
	\begin{align*}
	&\int_K(\nabla\times\Pi_K\bm u-\nabla\times\bm u)v \d V=\int_K(\Pi_K\bm u-\bm u)\nabla\times v\d V+\int_{\partial K}\bm n\times(\Pi_K\bm u-\bm u) v\d s=0.
	\end{align*}
	From Cauchy-Schwarz inequality, we derive
	\begin{align*}
	&\int_K(\nabla\times\bm u-I_K\nabla\times\bm u)(\nabla\times\Pi_K\bm u-I_K\nabla\times\bm u)\d V\\
	\leq& \left\|\nabla\times\bm u-I_K\nabla\times\bm u\right\|_K\left\| \nabla\times\Pi_K\bm u-I_K\nabla\times\bm u\right\|_K.
	\end{align*}
 Collecting the above three equations gives \eqref{pi_I}.
\end{proof}
\begin{lemma}\cite{Li2013Time}\label{hatv-v}
	Suppose that $\bm v $ and $\hat{\bm v}$ are related by the transformation \eqref{trans-1-1}. Then  for any $s\geq 0$, we have
	\begin{align*}
	&\left|\hat{\bm v}\right|_{s,\hat K}\leq C\left|\det(B_K)\right|^{-\frac{1}{2}}\left\|B_K\right\|^{s+1}\left\|\bm v\right\|_{s,K},\\
	&\big|\hat{\nabla}\times\hat{\bm v}\big|_{s,\hat K}\leq C\left|\det(B_K)\right|^{\frac{1}{2}}\left\|B_K\right\|^{s-1}\left\| \nabla\times \bm v\right\|_{s,K}.
	\end{align*}
\end{lemma}

\begin{theorem}\label{err-interp}
	If $\bm u$, $\nabla\times\bm u\in \bm H^s(K)$, $1+\delta\leq s\leq k$ with $\delta>0$, then we have the following error estimates for interpolation $\Pi_K$,
	\begin{align}
	&\left\|\bm u-\Pi_K\bm u\right\|_K+\left\|\nabla\times(\bm u-\Pi_K\bm u)\right\|_K\leq Ch^s(\left\|\bm u\right\|_{s,K}+\left\|\nabla\times\bm u\right\|_{s,K}),\label{inter-u}\\
	&	\left\|(\nabla\times)^2(\bm u-\Pi_K\bm u)\right\|_K\leq Ch^{s-1}\left\|\nabla\times\bm u\right\|_{s,K},
	\end{align}
	\begin{proof}
		We only prove the results for integer $s$ to avoid the technical complications.  We divide our proof in three steps.\\
		(i). We apply the transformation \eqref{trans-1-1} and Lemma \ref{hatpi} to derive
		\begin{align*}
		&\left\|\bm u-\Pi_K\bm u\right\|_K\\
		=&\left(\int_{\hat{K}}\left| B_K^{-T}(\hat{\bm u}-\widehat{\Pi_K\bm u})\right|^2\left|\det(B_K)\right|\d \hat V\right)^{\frac{1}{2}}\\
		\leq &\left|\det(B_K)\right|^{\frac{1}{2}}\left\|B_K^{-1}\right\|
		\left\|\hat{\bm u}-{\Pi_{\hat K}\hat{\bm u}}\right\|_{\hat K}.
		\end{align*}
	    Noting the fact that $\Pi_{\hat K}\hat {\bm p}=\hat{\bm p}$ when $\hat{\bm p}\in Q_{k-1}(\hat K)$, we obtain, with the help of Lemma \ref{wel-def-01} and Theorem 5.5 in \cite{Monk2003},
		\begin{align*}
		&\left\|\hat{\bm u}-{\Pi_{\hat K}\hat{\bm u}}\right\|_{\hat K}
		=\left\|\left(I-\Pi_{\hat K}\right)(\hat{\bm u}+\hat{\bm p})\right\|_{\hat K}\\
		\leq & \inf_{\hat {\bm p}}C\left(\left\|\hat{\bm u}+\hat{\bm p}\right\|_{s,\hat K}+\left\|\nabla\times(\hat{\bm u}+\hat{\bm p})\right\|_{s,\hat K}\right)\\
		\leq &C\left(\left|\hat{\bm u}\right|_{s,\hat K}+\left|\nabla\times\hat{\bm u}\right|_{s,\hat K}\right).
		\end{align*}
       Collecting the above two equations and using Lemma \ref{hatv-v}   leads to
       \[\left\|\bm u-\Pi_K\bm u\right\|_K\leq Ch^{s}(\left\|\bm u \right\|_{s,K}+\left\| \nabla\times \bm u\right\|_{s,K}).\]
       (ii). We use triangle inequality and Lemma \ref{lem2.5} to have
		\begin{align*}
		&\left\|\nabla\times(\bm u-\Pi_K\bm u)\right\|_K\\
		\leq &\left\|\left(I-I_K\right)\nabla \times\bm u\right\|_K+\left\|I_K\nabla\times\bm u-\nabla\times\Pi_K\bm u\right\|_K\\
		\leq &2\left\|\left(I-I_K\right)\nabla \times\bm u\right\|_K,
		\end{align*}
		which, together with the error estimate of Lagrange interpolation, leads to
		 \[\left\|\nabla\times(\bm u-\Pi_K\bm u)\right\|_K\leq Ch^{s}\left\| \nabla\times \bm u\right\|_{s,K}.\]
		 (iii). Applying the triangle inequality and noting the fact that $\left\|\nabla\times\varphi\right\|=\left|\varphi\right|_1$, we arrive at
			\begin{align*}
		&\left\|(\nabla\times)^2(\bm u-\Pi_K\bm u)\right\|_K\\
		\leq&\left\|(\nabla\times)^2\bm u -\nabla\times I_K\nabla\times\bm u\right\|_K+\left\|\nabla\times I_K\nabla\times\bm u-(\nabla\times)^2\Pi_K\bm u)\right\|_K\\
		\leq &\left|\nabla\times\bm u - I_K\nabla\times\bm u\right|_{1,K}+Ch^{-1}\left\| I_K\nabla\times\bm u-\nabla\times\Pi_K\bm u)\right\|_K,
		\end{align*}
	where we have used $\left\|\nabla\times I_K\nabla\times\bm u-(\nabla\times)^2\Pi_K\bm u)\right\|_K\leq Ch^{-1}\left\| I_K\nabla\times\bm u-\nabla\times\Pi_K\bm u)\right\|_K$. In fact,
	\begin{align*}
	&\left\|\nabla\times I_K\nabla\times\bm u-(\nabla\times)^2\Pi_K\bm u)\right\|_K^2\\
	\leq&\frac{\left\|B_K\right\|^2}{\det(B_K)}\int_{\hat K}\left(\hat\nabla\times\left(\left({{I}_{K} \nabla \times {\bm u}}-{\nabla\times\Pi_{ K}{\bm u}}\right)\circ F_K\right)\right)^2\d \hat V\\
	\leq &C\frac{\left\|B_K\right\|^2}{\det(B_K)}\int_{\hat K}\left(\left({{I}_{K} \nabla \times {\bm u}}-{\nabla\times\Pi_{ K}{\bm u}}\right)\circ F_K\right)^2\d \hat V\\
	\leq &C\frac{\left\|B_K\right\|^2}{\det(B_K)^2}\int_{ K}\left({{I}_{K} \nabla \times {\bm u}}-{\nabla\times\Pi_{ K}{\bm u}}\right)^2\d  V\\
	\leq & Ch^{-2}\left\|{{I}_{K} \nabla \times {\bm u}}-{\nabla\times\Pi_{ K}{\bm u}}\right\|_K^2.
	\end{align*}
   Combining the above two equations and applying Lemma \ref{lem2.5}, we acquire
		\[\left\|(\nabla\times)^2(\bm u-\Pi_K\bm u)\right\|_K
	\leq  Ch^{s-1}\left\|\nabla\times\bm u\right\|_{s,K}.\]
	\end{proof}
\end{theorem}

\begin{remark}
	 The equation (5.42) in \cite{Monk2003}, which is for the  $H(\tc)$ interpolation, has the same form with our estimate \eqref{inter-u}. However, \eqref{inter-u} is for the $H^2(\tc)$ interpolation, hence, we proved it again.
\end{remark}

\section{$H^2$(\text{curl})-conforming elements on triangles}\label{trielem}
We also construct an $H^2(\tc)$-conforming finite element on a triangle. Proceeding as Section \ref{rectelem}, we can prove the same theoretical results. Thus, in this section, we only introduce the definition of the finite elements. To this end, we need to define a special space of polynomial  $\mathcal{R}_k$:
$$\mathcal{R}_k=(P_{k-1})^2\oplus{\Phi}_k,\,\ \Phi_k=\{\bm p \in (\widetilde{P}_k)^2:\bm x \cdot\bm p =0\},$$
where $\widetilde{P}_k$ is the space of homogeneous polynomial of degree $k$.
\begin{definition}\label{2def2} For any integer $k \geq 4$, the ${H}^2$(\text{curl})-conforming element is defined by the triple:
\begin{equation*}
\begin{split}
&\hat{K}\;\text{is the reference triangle},\\
&P_{\hat{K}} = \mathcal{R}_k,\\
&\Sigma_{\hat{K}} = \bm M_{\hat{p}}(\hat {\bm u}) \cup \bm M_{\hat{e}}(\hat {\bm u}) \cup \bm M_{\hat{K}}(\hat {\bm u}),
\end{split}
\end{equation*}
where $\bm M_{\hat{p}}(\hat {\bm u})$ is the set of DOFs given on all vertex nodes and edge nodes $\hat{p}_{i}$:
\begin{equation}\label{2def2}
\bm M_{\hat{p}}(\hat {\bm u})=\left\{(\hat{\nabla}\times \hat{\bm u})(\hat{p_{i}}), i=1,\;2,\;\cdots\;,3(k-1) \right\},
\end{equation}
with the points $p_{i}$ chosen at : $3$ vertex nodes and $(k-2)$ distinct nodes in each edge,\\
$\bm M_{\hat{e}}(\hat {\bm u})$ is the set of DOFs given on all edges $\hat{e}_i$ of $\hat{K}$, each with the unit  tangential vector $\hat{\bm \tau}_i$:
\begin{equation}\label{2def3}
\bm M_{\hat{e}}(\hat {\bm u})= \left\{\int_{\hat{e}_i}\hat {\bm u}\cdot \hat{\bm \tau}_i \hat{q}\mathrm d\hat{s},\ \forall \hat{q}\in P_{k-1}(\hat{e}_i),\;i=1,2,3 \right\},
\end{equation}
and $M_{\hat{K}}(\hat {\bm u})$ is the set of DOFs given on the element $\hat{K}$:
\begin{align}\label{2def4}
\bm M_{\hat{K}}(\hat {\bm u})=\left\{\int_{\hat{K}} \hat{\bm u}\cdot \hat{\bm q} \mathrm d \hat V,\ \forall \hat{\bm q}\in  \left(P_{k-5}\right)^2\oplus\widetilde{P}_{k-5}\hat{\bm x}\oplus\widetilde{P}_{k-4}\hat{\bm x}\oplus
\widetilde{P}_{k-3}\hat{\bm x}\right\}.
\end{align}
\end{definition}
The total number of DOFs is
		\[\dim(P_{\hat{K}})= 2 \dim (P_{k-1})+\dim \Phi_k = k(k+1)+k=k(k+2).\]
		Since $k\geq 4$, the minimum number of DOFs  is 24.
\begin{lemma}\label{trilemm-2}
		Suppose that the function $\bm u$ and the unit  tangential vector $\bm \tau$ to $\partial K$ are defined by the transformations \eqref{trans-1-1} and \eqref{trans-3}. Suppose also that the DOFs of a function $\bm u$ on $K$ are given by
	\begin{align}
	\bm M_{{p}}({\bm u})=&\Big\{\det(B_K)({\nabla}\times {\bm u})(p_{i}),\;  i=1,\;2,\cdots,4(k-1) \Big\},\label{triLdef2}\\
	\bm M_{{e}}( {\bm u})=& \left\{\int_{e_i} {\bm u}\cdot{\bm \tau}_i {q}\d{s},\ \forall {q}\in P_{k-1}({e}_i), i=1,2,3,4 \right\}\label{triLdef3},\\
	\bm M_{{K}}(\bm { u})=&\left\{\int_{{K}} {\bm u}\cdot{\bm q}\d V,\ \forall \bm q \circ F_K = \frac{B_K\hat{\bm q}}{\det(B_K)},\;\hat{\bm q}\in  \left(P_{k-5}\right)^2\oplus\widetilde{P}_{k-5}\hat{\bm x}\oplus\widetilde{P}_{k-4}\hat{\bm x}\oplus
	\widetilde{P}_{k-3}\hat{\bm x}\right\}.\label{triLdef4}
	\end{align}
	Then the DOFs for $\hat{\bm u}$ on $\hat K$ and for $\bm u$ on $K$ are identical.
\end{lemma}

\section{Applications}\label{app}
In this section, we use the $H^2$(curl)-conforming finite elements developed in Section \ref{rectelem} and \ref{trielem} to solve the quad-curl problem which is introduced as: For $\bm  f\in H(\td^0;\Omega)$, find
\begin{equation}\label{prob1}
\begin{split}
(\nabla\times)^4\bm u&=\bm f,\quad \text{in}\;\Omega,\\
\nabla \cdot \bm u &= 0, \quad \text{in}\;\Omega,\\
\bm u\times\bm n&=0,\quad \text{on}\;\partial \Omega,\\
\nabla \times \bm u&=0,\quad \text{on}\;\partial \Omega,
\end{split}
\end{equation}
where $\Omega \in\mathbb{R}^2$ is   Lipschitz domain and  $\bm n$ is the unit outward normal vector to $\partial \Omega$.
For the sake of satisfying divergence-free condition, we adopt mixed methods where the constraint $\nabla\cdot\bm u=0$ in \eqref{prob1} is satisfied in a weak sense by introducing an auxiliary unknown $p$ and employing a mixed  variational formulation:
Find $(\bm u; p)\in H_0^2(\tc;\Omega)\times H^1_0(\Omega)$,  s.t.
\begin{equation}\label{prob22}
\begin{split}
a(\bm u,\bm v) + b(\bm v,p)&=(\bm f, \bm v),\quad \forall \bm v\in H^2_0(\tc;\Omega),\\
b(\bm u,q)&=0,\quad \hspace{0.5cm}\forall q\in H^1_0(\Omega),
\end{split}
\end{equation}
where
\begin{align*}
&a(\bm u,\bm v)=((\nabla\times)^2 \bm u, (\nabla\times)^2 \bm v), \\
&b(\bm v,p)=(\bm v,\nabla p ).
\end{align*}
The well-posedness of the variational problem can be found in \cite{quadcurlWG}. Due to $\bm  f\in H(\td^0;\Omega)$, $p=0$.

Let \,$\mathcal{T}_h\,$ be a partition of the domain $\Omega$
consisting of rectangles or triangles. For every element $K \in
\mathcal{T}_h$, we denote by $h_K$ its diameter.  And we denote by $h$ the mesh size of $\mathcal {T}_h$. We define
\begin{eqnarray*}
	&&  V_h=\{\bm{v}_h\in H^2(\text{curl};\Omega):\ \bm v_h|_K\in Q_{{k-1},k}\times Q_{k,{k-1}}\;\text{or}\;\mathcal{R}_k,\ \forall K\in\mathcal{T}_h\}.\\
	&&   V^0_h=\{\bm{v}_h\in V_h,\ \bm{n} \times \bm{v}_h=0\ \text{and}\ \nabla\times  \bm{v}_h = 0 \ \text {on} \ \partial\Omega\},\\
	&&  S_h=\{{w}_h\in H^1(\Omega):\  w_h|_K\in Q_{k}\;\text{or}\;P_k\},\\
	&&  S^0_h=\{{w}_h\in W_h,\;{w}_h|_{\partial\Omega}=0\}.
\end{eqnarray*}

The $H^2$(curl)-conforming finite element method seeks $(\bm u_h;p_h)\in V^0_h\times S^0_h$,  s.t.
\begin{equation}\label{prob3}
\begin{split}
a(\bm u_h,\bm v_h) + b(\bm v_h,p_h)&=(\bm f, \bm v_h),\quad \forall \bm v_h\in V^0_h,\\
b(\bm u_h,q_h)&=0,\quad \hspace{0.7cm}\forall q_h\in S^0_h.
\end{split}
\end{equation}

\begin{lemma}
	The discrete problem  \eqref{prob3} has  a unique solution and $p_h=0$.
\end{lemma}
\begin{proof}
	Firstly, we define a space $$X_h=\{\bm u_h\in V^0_h \ | \ b(\bm u_h,q_h)=0,\ \  \text{for all}\ \ q_h\in S^0_h\}.$$
	By the Poincar\'e inequality, the boundary condition, and the discrete Friedrichs inequality \cite{Sun2016A}, we can deduce that $a(\cdot,\cdot)$ is coercive  on $X_h$, i.e., for $\bm u_h\in X_h$,
	\begin{equation}\label{prob5}
	\begin{split}
	a(\bm u_h,\bm u_h)&=((\nabla\times)^2 \bm u_h, (\nabla\times)^2 \bm u_h) \geq \alpha(\|\bm u_h\|+\|\nabla\times \bm u_h\|+ \|(\nabla\times)^2 \bm u_h\|),
	\end{split}
	\end{equation}
	where $\alpha$ is a positive constant.
	Here  we have used the  fact $\|(\nabla\times)^2 \bm u_h\|=|\nabla\times \bm u_h|_{1}\geq C\|\nabla\times \bm u_h\|$.
	Furthermore, we check the Babu$\check{\text{s}}$ka-Brezzi condition by taking $\bm v_h=\nabla p_h $, then
	\begin{equation}\label{prob6}
	\begin{split}
	\sup_{\bm v_h\in V^0_h}\frac{|b(\bm v_h,p_h)|}{\|\bm v_h\|_{H^2(\tc;\Omega)}}\geq\frac{|(\nabla p_h,\nabla p_h)|}{\|\nabla p_h\|_{H^2(\tc;\Omega)}}\geq\|\nabla p_h\|\geq C\| p_h\|_{H^1(\Omega)}.
	\end{split}
	\end{equation}
	The $X_h$-coercivity and Babu\v{s}ka-Brezzi condition are satisfied, then the  problem \eqref{prob3} has a unique solution.
	Moreover, by letting $\bm u_h = \nabla p_h$ in the second equation of \eqref{prob3}, we have $\|\nabla p_h\|=0.$ Combined with the boundary condition of $p_h$, we arrive at $p_h=0$.
\end{proof}

Before giving the error estimate of $(\bm u_h;p_h)$, we first define a global interpolation.
For $\bm u \in \bm H^{1/2+\delta}(\Omega)$ and $ \nabla \times \bm u \in H^{1+\delta}(\Omega)$ with $\delta >0$, the global inerpolation $\Pi_h\bm u\in V_h$ is defined element by element using
\[(\Pi_h\bm u)|_K=\Pi_K(\bm u|_K).\]
\begin{theorem}\label{conv-curlcurl}
	Assume $(\bm u; p)\in H^2_0(\mathrm{curl};\Omega)\times H^1_0(\Omega)$ is the solution of \eqref{prob22} with $p=0$ and $(\bm u_h; p_h)\in V^0_h\times S_h^0$ is the solution of \eqref{prob3} with $p_h=0$. Then, for $\bm u\in \bm H^s(\Omega)$ and $\nabla\times\bm u\in H^s(\Omega)$ $(1+\delta\leq s\leq k\;\text{with}\;\delta>0)$, we have
	\begin{equation*}
	\begin{split}
	\|\bm u-\bm u_h\|_{H^2(\mathrm{curl};\Omega)} \leq C h^{s-1}\left(\left\|\bm u\right\|_s+\left\|\nabla\times\bm u\right\|_s\right).
	\end{split}
	\end{equation*}
\end{theorem}
\begin{proof}
	Applying Theorem 2.45 in \cite{Monk2003}, we have
	\begin{equation*}
	\begin{split}
	\|\bm u-\bm u_h\|_{H^2(\mathrm{curl};\Omega)} &\leq \inf_{v_h\in V_h}\|\bm u-\bm v_h\|_{H^2(\mathrm{curl};\Omega)} \leq
	\|\bm u-\Pi_h \bm u\|_{H^2(\mathrm{curl};\Omega)}\leq C h^{s-1}(\left\|\bm u\right\|_s+\left\|\nabla\times\bm u\right\|_s).
	\end{split}
	\end{equation*}
\end{proof}
We consider the following auxiliary problem: Find $\bm w\in X$, s.t.
\begin{equation}\label{auxiprob1}
a(\bm v,\bm w) =((\nabla\times)^2(\bm u-\bm u_h), \bm v)\quad \forall \bm v\in X,
\end{equation}
where $X$ is defined in \eqref{spac-02}.
Moreover, we assume
\begin{align}\label{regularity}
\left\|\nabla\times\bm w\right\|_2\leq \left\|\nabla\times(\bm u-\bm u_h)\right\|.
\end{align}
\begin{theorem}\label{conv-u}
	Under the assumptions of Theorem \ref{conv-curlcurl} and \eqref{regularity}, we have
	\begin{equation*}
	\begin{split}
	\|\nabla\times(\bm u-\bm u_h)\| \leq C h^{s}(\left\|\bm u\right\|_s+\left\|\nabla\times\bm u\right\|_s),\quad 1+\delta \leq s\leq k,\;\delta>0.
	\end{split}
	\end{equation*}
\end{theorem}
\begin{proof}
	Using Lemma \ref{Helm}, we can decompose $\bm u-\bm u_h$ as $\bm u-\bm u_h=\bm \pi_1+\bm \pi_2$ with $\bm \pi_1\in X$ and $\bm \pi_2\in \nabla H_0^1(\Omega)$. We take $\bm v=\bm \pi_1$ to get
	\[a(\bm u-\bm u_h, \bm w)=a(\bm\pi_1, \bm w)=((\nabla\times)^2(\bm u-\bm u_h),\bm{\pi}_1)=\left\|\nabla\times(\bm u-\bm u_h)\right\|^2.\]
	Due to the fact that $p=p_h=0$, we arrive at
	\[a(\bm u-\bm u_h, \bm w_h)=0,\quad \forall \bm w_h\in X_h,\]
	which, together with Theorem \ref{err-interp}  and \eqref {regularity},  leads to
	\begin{align*}
	&\left\|\nabla\times(\bm u-\bm u_h)\right\|^2=a(\bm u-\bm u_h, \bm w-\bm w_h)\\
	\leq & C\left\|\bm u-\bm u_h\right\|_{H^2(\tc;\Omega)}\inf_{\bm w_h\in X_h}\left\| (\nabla\times)^2(\bm w-\bm w_h)\right\|\\
	= & C\left\|\bm u-\bm u_h\right\|_{H^2(\tc;\Omega)}\inf_{\bm w_h\in V_h^0}\left\| (\nabla\times)^2(\bm w-\bm w_h)\right\|\\
	\leq & Ch\left\|\bm u-\bm u_h\right\|_{H^2(\tc;\Omega)}\left\|\nabla\times\bm w\right\|_2\\
	\leq & Ch\left\|\bm u-\bm u_h\right\|_{H^2(\tc;\Omega)}\left\|\nabla\times(\bm u-\bm u_h)\right\|,
	\end{align*}
	Thus, we complete the proof.
\end{proof}
\begin{theorem}\label{conv-curlu}
	Under the assumptions of Theorem \ref{conv-u}, we have
	\begin{equation*}
	\begin{split}
	\|\bm u-\bm u_h\| \leq C h^{s}(\left\|\bm u\right\|_s+\left\|\nabla\times\bm u\right\|_s),\quad 1+\delta \leq s\leq k,\;\delta>0.
	\end{split}
	\end{equation*}
\end{theorem}
\begin{proof}
	 This theorem can be proved by the same method as employed in \cite[Theorem 6]{Sun2016A}.
\end{proof}
\section{numerical experiments}
{We compare our finite element method (FEM) with a  mixed finite element method (MFEM) \cite{Sun2016A} and Hodge decomposition (HD) method \cite{Brenner2017Hodge}.

From the Table \ref{tabcompare1} and \ref{tabcompare2}, we can see clearly the comparison among the three methods in this case. To further show the effectiveness of our method, we take $\Omega=(0,1)\times(0,1)$ and partition it into $N^2$ squares. On the one hand, the total DOFs of the FEMs and MFEMs  are $M_1$ and  $M_1+\Delta_1$ respectively, where $M_1=2(N+1)^2+6k(N+1)N+(3k^2-2k)N^2$,
$\Delta_1=2N^2(k^2-2k-1)-4N-1$\;($\Delta_1>0$ when $k\geq3,N\geq2$). In other words, the linear system that needs to be solved in our method is smaller than the one in MFEMs, especially when $N$ is a large number. On the other hand, even if the HD method enjoys the advantage of smaller computation cost, the convergence order of it depends not only on the regularity of the solutions but also the regularity of the domain, e.g., the highest convergence order is 2 when $\Omega$  is a square. Furthermore, we partition $\Omega$ into $2N^2$ uniform triangles of regular pattern. The total DOFs of the FEMs and MFEMs in this case are $M_2$ and $M_2+\Delta_2$, where $M_2=2(N+1)^2+3k(3N^2+2N)+(3k^2-5k)N^2$,
$\Delta_2=2N^2(k^2-k-4)-8N-1$ ($\Delta_2>0$ when $k\geq 3,N\geq3$).

\begin{table}[!ht]
		\caption{Rectangle mesh: Comparison among FEMs, MFEMs and HD with $\Delta_1=2N^2(k^2-2k-1)-4N-1$ ($\Delta_1>0$, when $k\geq 3,N\geq2$ ).} \label{tabcompare1}
		\centering
		\begin{adjustwidth}{0cm}{0cm}
			\resizebox{!}{2.1cm}
			{
				\begin{tabular}{cccccccccc}
					\hline
					\multicolumn{1}{c}{\multirow{2}{1.5cm}{Methods}}& \multicolumn{1}{c}{\multirow{2}{1cm}{FEMs}} &\multicolumn{1}{c}{\multirow{1}{0cm}{}}  &\multicolumn{1}{c}{\multirow{2}{1cm}{MFEMs}}&\multicolumn{1}{c}{\multirow{1}{0cm}{}}  &\multicolumn{3}{c}{HD ($k\geq 1$)}\\
					\cline{6-8}\multicolumn{1}{c}{}&\multicolumn{1}{c}{($k\geq 2$)}  &\multicolumn{1}{c}{}  &\multicolumn{1}{c} {($k\geq 1$)} &\multicolumn{1}{c}{} &\multicolumn{3}{c} {$LS_i(1\leq i\leq4)$} \\ \hline
					
					elements&$Q_{k,k+1}\times Q_{k+1,k}\times Q_{k+1}$&&$(Q_{k-1,k}\times Q_{k,k-1})^2\times Q_k$&&\multicolumn{3}{c}{$Q_k$}\\\hline
					local node DOFs&$4+4$&&$0+4$&&\multicolumn{3}{c}{$4$}\\\hline
					local edge DOFs&$8k+4k$&&$8k+4(k-1)$&&\multicolumn{3}{c}{$4(k-1)$}\\\hline
					local element  DOFs&$2k(k-1)+k^2$&&$4k(k-1)+(k-1)^2$&&\multicolumn{3}{c}{$(k-1)^2$}\\\hline
					convergence order in $H^2$(curl)-norm&$k$&&$k$&&\multicolumn{3}{c}{$\min\{\frac{\pi}{\omega}-\varepsilon,k\}$}\\\hline
					convergence order in $L^2$-norm&$k+1$&&$k$&&\multicolumn{3}{c}{$\min\{\frac{\pi}{\omega}-\varepsilon,k\}$}\\\hline
					regularity&solution&&solution&&\multicolumn{3}{c}{solution and domain}\\\hline
				
					global DOFs on a square with $N^2$ uniform rectangles&$M_1$&&$M_1+\Delta_1$&&\multicolumn{3}{c}{--}\\\hline
				\end{tabular}
			}
		\end{adjustwidth}
	\end{table}

\begin{table}[!ht]
	\caption{Triangular mesh: Comparison among FEMs, MFEMs and HD with $\Delta_2=2N^2(k^2-k-4)-8N-1$ ($\Delta_2>0$, when $k\geq 3,N\geq3$).} \label{tabcompare2}
	\centering
	\begin{adjustwidth}{0cm}{0cm}
		\resizebox{!}{2.3cm}
		{
			\begin{tabular}{cccccccccc}
				\hline
				\multicolumn{1}{c}{\multirow{2}{1.5cm}{Methods}}& \multicolumn{1}{c}{\multirow{2}{1cm}{FEMs}} &\multicolumn{1}{c}{\multirow{1}{0cm}{}}  &\multicolumn{1}{c}{\multirow{2}{1cm}{MFEMs}}&\multicolumn{1}{c}{\multirow{1}{0cm}{}}  &\multicolumn{3}{c}{HD ($k\geq 1$)}\\
				\cline{6-8}\multicolumn{1}{c}{}&\multicolumn{1}{c}{($k\geq 3$)}  &\multicolumn{1}{c}{}  &\multicolumn{1}{c} {($k\geq 1$)} &\multicolumn{1}{c}{} &\multicolumn{3}{c} {$LS_i(1\leq i\leq 4)$}\\ \hline
				
				elements&$R_{k+1}\times P_{k+1}$&&$(R_k)^2\times P_k$&&\multicolumn{3}{c}{$P_k$}\\\hline
					local node DOFs&$3+3$&&$0+3$&&$6$\\\hline
				local edge DOFs&$6k+3k$&&$6k+3(k-1)$&&$6(k-1)$\\\hline
				local element  DOFs&$k(k-2)+\frac{k^2-k}{2}$&&$2k(k-1)-\frac{(k-1)(k-2)}{2}$&&$(k-1)(k-2)$\\\hline
				
				convergence order in $H^2$(curl)-norm&$k$&&$k$&&\multicolumn{3}{c}{$k$}\\\hline
				convergence order in $L^2$-norm&$k+1$&&$k$&&\multicolumn{3}{c}{$k$}\\\hline
				regularity&solution&&solution&&\multicolumn{3}{c}{solution and domain}\\\hline
					global DOFs on a square with $2N^2$ uniform triangles&$M_2$&&$M_2+\Delta_2$&&\multicolumn{3}{c}{--}\\\hline
			\end{tabular}
		
		}
	\end{adjustwidth}
\end{table}


We further use several numerical examples to show the effectiveness of the  $H^2(\tc)$-conforming finite element method introduced in Section  \ref{app} and to verify the theoretical findings associated with the method.

\begin{example}\label{ex1}
\end{example}
We consider the problem \eqref{prob1} on a unit square $\Omega=(0,1)\times(0,1)$ with exact solution
\begin{equation}
\bm u=\left(
\begin{array}{c}
3\pi\sin^3(\pi x)\sin^2(\pi y)\cos(\pi y) \\
-3\pi \sin^3(\pi y)\sin^2(\pi x)\cos(\pi x)\\
\end{array}
\right).
\end{equation}

Then  the source term $\bm f$ can be obtained by a simple calculation. We have pointed out that $p=0$ if $\nabla\cdot \bm f=0$.

We first apply uniform rectangular partitions. Varying $h$ from $\frac{1}{40}$ to $\frac{1}{80}$, Table \ref{tab1} illustates the errors and convergence rates of $\bm u_h$ in several different norms. From the table, we can deduce that the numerical results coincide with the theoretical findings, which confirm the correctness of our analysis.

Our numerical examples are also conducted on a nonuniform triangle mesh. The numerical results are presented in Table \ref{tab2}, which indicates again the correctness of our numerical scheme and the corresponding error estimates.
\begin{table}[h]
	\centering
	\caption{Example \ref{ex1}: numerical results using the lowest-order $H^2(\tc)$ elements on rectangles} \label{tab1}
	\begin{tabular}{cccccccc}
		\hline
		$n$ &$\left\|\bm u-\bm u_h\right\|$& rates&$\left\|\nabla\times(\bm u-\bm u_h)\right\|$& rates&$\left\|\nabla\times\nabla\times(\bm u-\bm u_h)\right\|$& rates\\
		\hline
		$40 \times 40$ &2.5485449381e-05&    &1.1472108502e-03&&2.9760181442e-01& \\
		$50\times 50$&1.2854795005e-05 &3.0670 &5.8764134991e-04& 2.9979& 1.9050383117e-01 &1.9991\\
		$60\times 60$&7.3774307075e-06& 3.0457 &3.4015484126e-04 &2.9986 &1.3230890722e-01 &1.9994\\
		$70\times 70$&4.6222504985e-06& 3.0330 &2.1424041027e-04 &2.9990& 9.7213001130e-02 &1.9996\\
		$80\times 80$& 3.0862396038e-06& 3.0250& 1.4353829491e-04 &2.9993 &7.4431912057e-02 &1.9997\\
		\hline
	\end{tabular}
\end{table}

\begin{table}[h]
	\centering
	\caption{Example \ref{ex1}: numerical results  using the lowest-order $H^2(\tc)$ elements on triangles} \label{tab2}
	\begin{tabular}{cccccccc}
		\hline
		$h$ &$\left\|\bm u-\bm u_h\right\|$& rates&$\left\|\nabla\times(\bm u-\bm u_h)\right\|$& rates&$\left\|\nabla\times\nabla\times(\bm u-\bm u_h)\right\|$& rates\\
		\hline
		$1/10$ &6.7345327239e-05&    &2.7173652598e-03&&3.2457589560e-01& \\
		$1/20$&3.8515300864e-06&4.1281&1.6843052463e-04&4.0120&4.0729581739e-02&2.9944\\
		$1/40$&2.3512856825e-07& 4.0339&1.0490867329e-05&4.0049&5.0982984036e-03&2.9980\\
		$1/80$&1.4617020486e-08&4.0077&6.5478343914e-07&4.0020&6.3766061097e-04&2.9992\\
		\hline
	\end{tabular}
\end{table}

\begin{example}\label{ex2}
\end{example}
We also consider the problem \eqref{prob1} on an L-shape domain $\Omega=(0,1)\times(0,1)\slash[0.5,1)\times(0,0.5]$ with source term $\bm f=(1,1)^T$.

We adopt the graded mesh introduced in \cite{Li2013LNG} with a grading parameter $\kappa$ (see Figure \ref{kap5} and \ref{kap245}). When $\kappa=0.5$, the mesh is uniform. Table \ref{tab4} illustrates errors and convergence rates of $\bm u_h$ in this case.  Due to the singularity of the domain, the rates of convergence deteriorate.  And rates of $4/3$ and $2/3$ or so can be observed from the table. When $\kappa=0.245$, the mesh is dense near the singular point $(0,0)$. The numerical results in this case are shown in Table \ref{tab5}, the convergence rates are improved significantly.

\begin{figure}
	\centering
\includegraphics[width=0.32\linewidth, height=0.2\textheight]{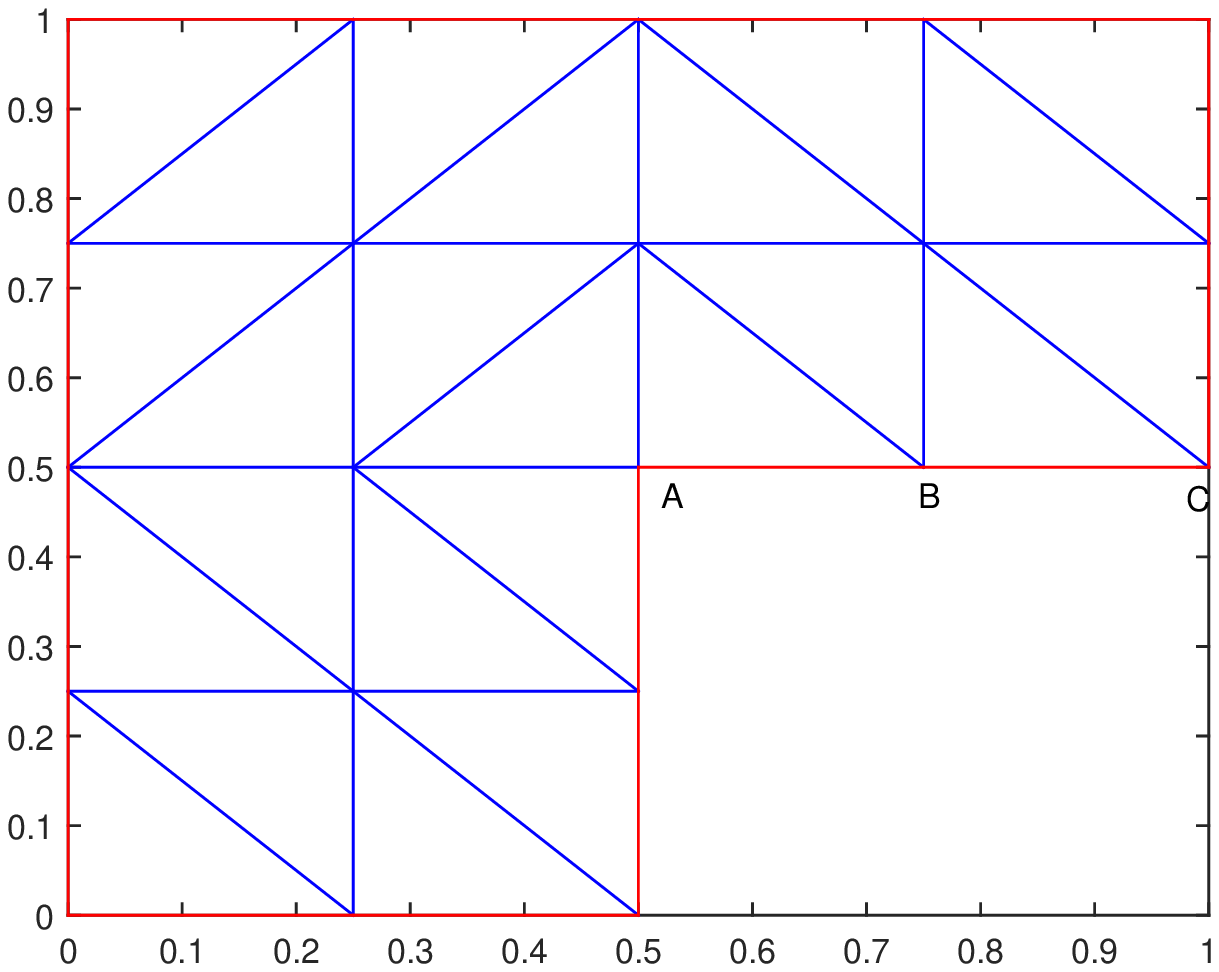}
\includegraphics[width=0.32\linewidth, height=0.2\textheight]{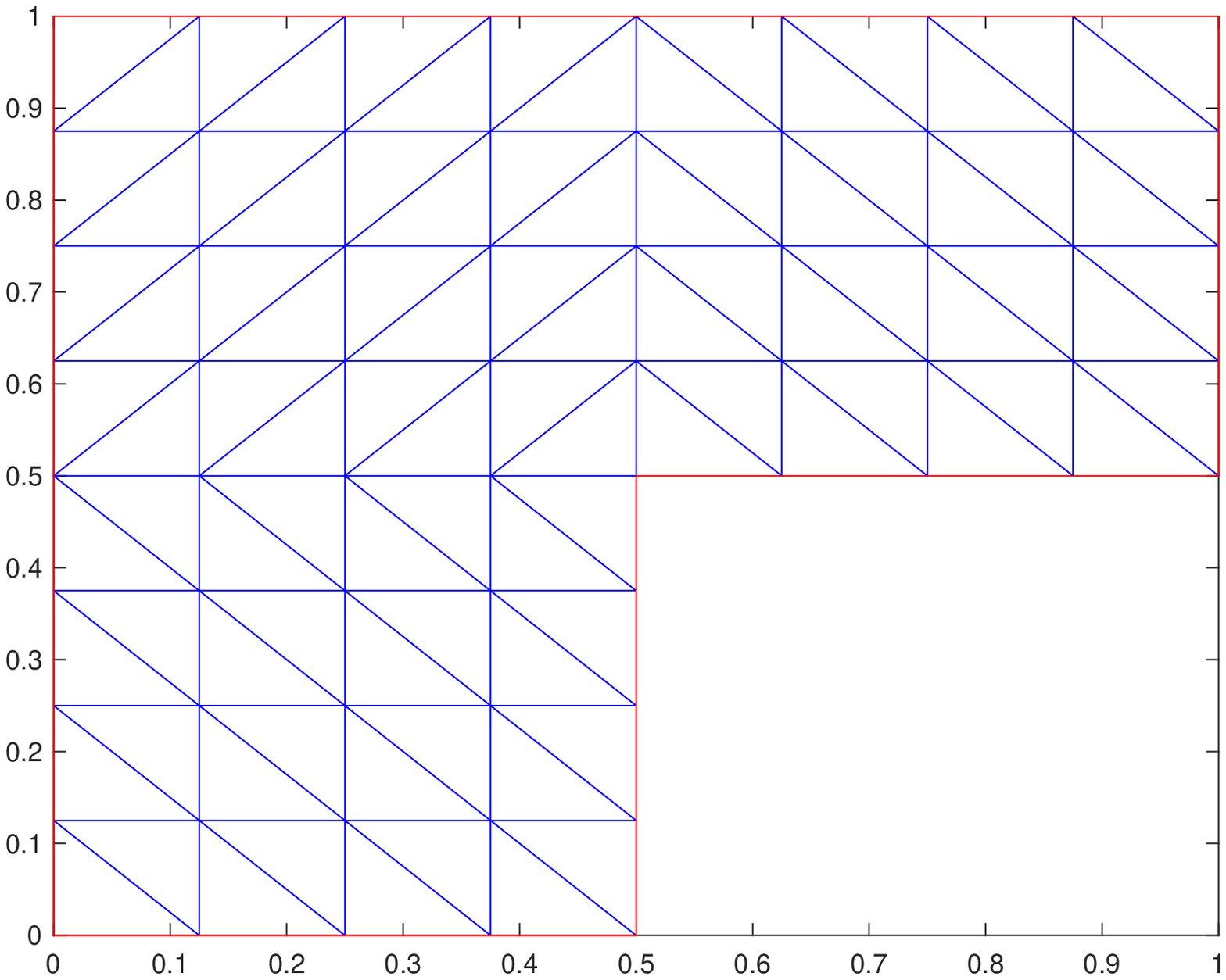}
\includegraphics[width=0.32\linewidth, height=0.2\textheight]{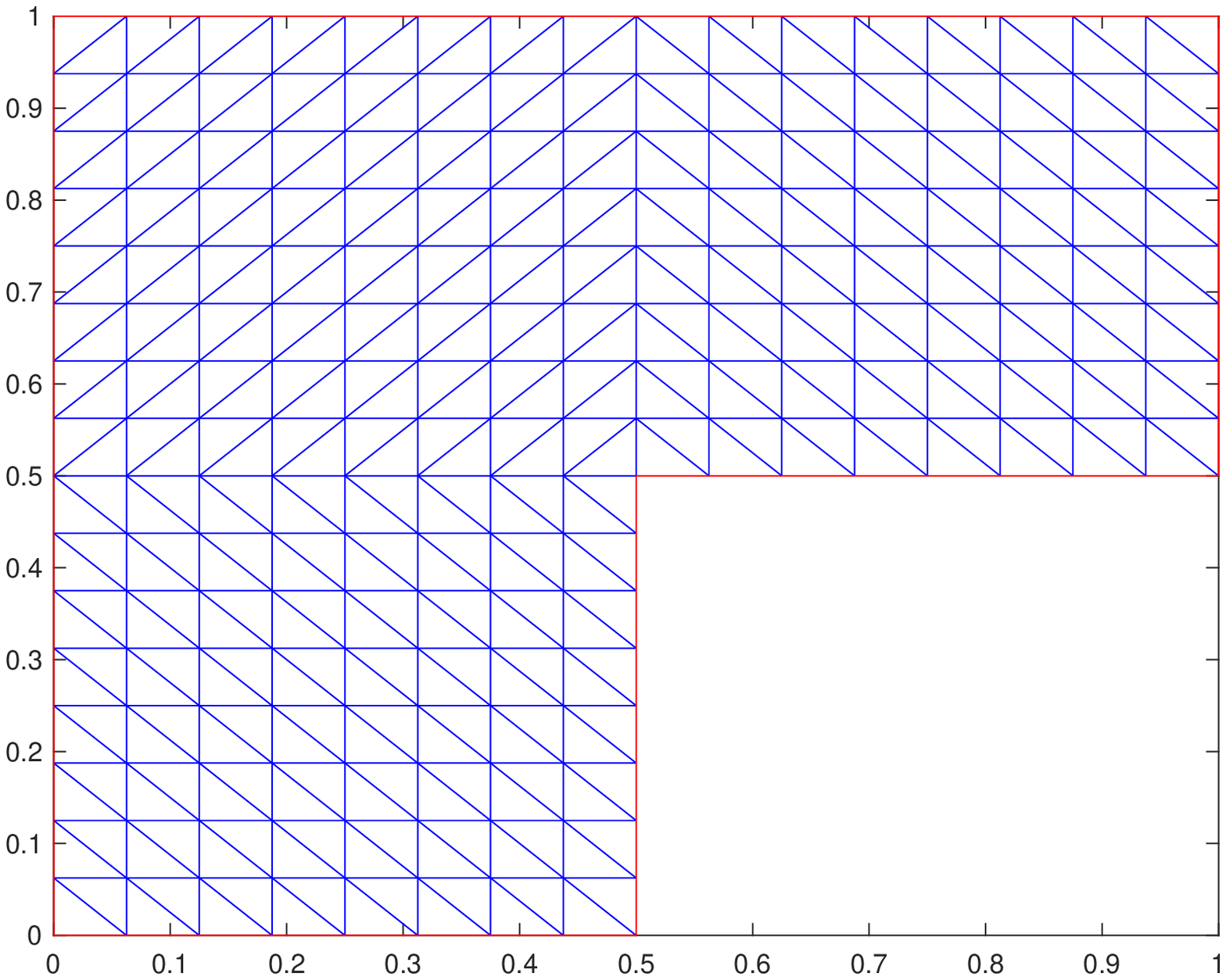}\\
	$n=1$\quad\quad\quad\quad\quad\quad\quad\quad\quad\quad\quad $n=2$\quad\quad\quad\quad\quad\quad\quad\quad\quad\quad\quad\quad\quad$n=3$
\caption{graded mesh with $\kappa=\frac{|AB|}{|AC|}=0.5$}\label{kap5}
\end{figure}
\begin{figure}
	\centering
	\includegraphics[width=0.32\linewidth, height=0.2\textheight]{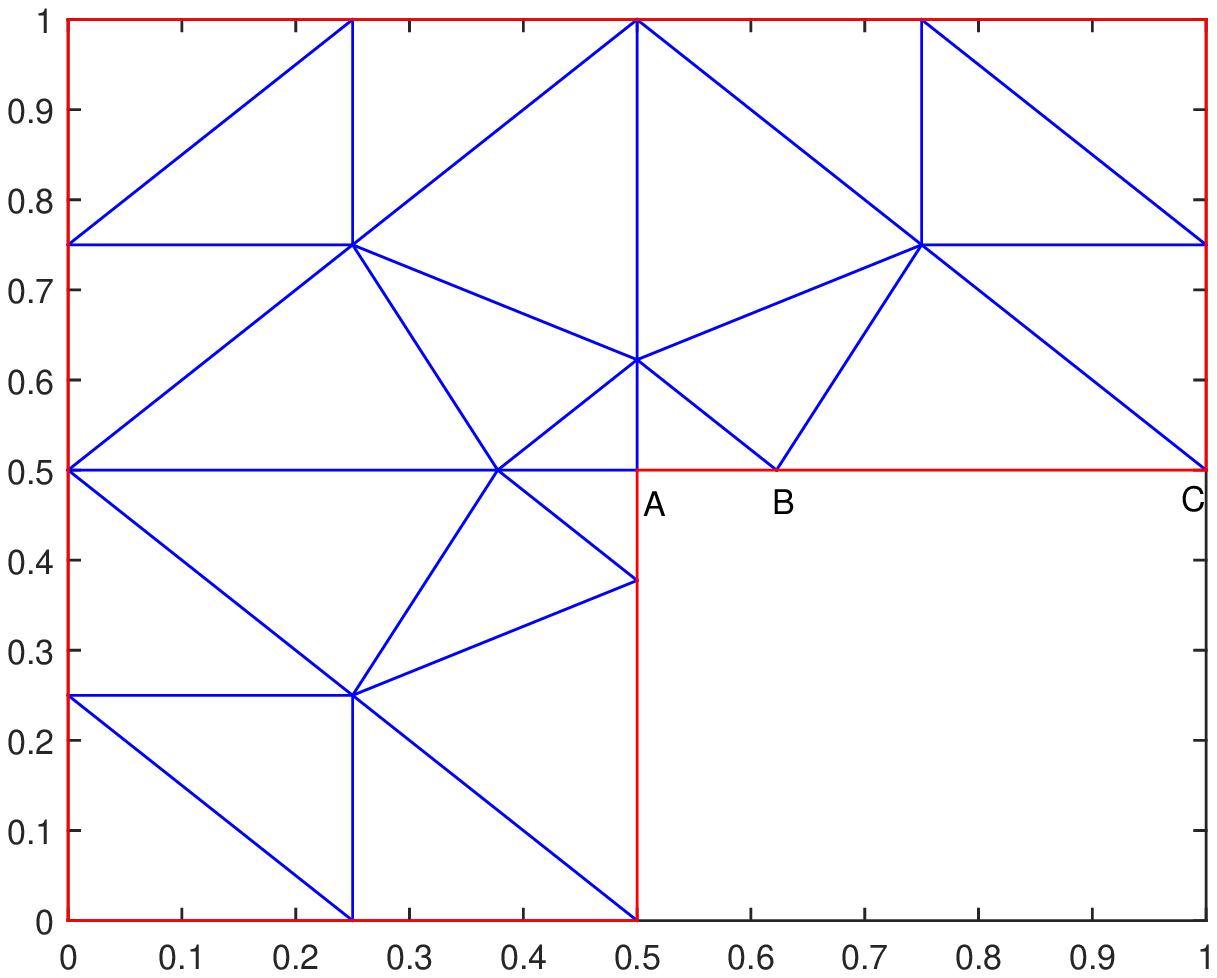}
	\includegraphics[width=0.32\linewidth, height=0.2\textheight]{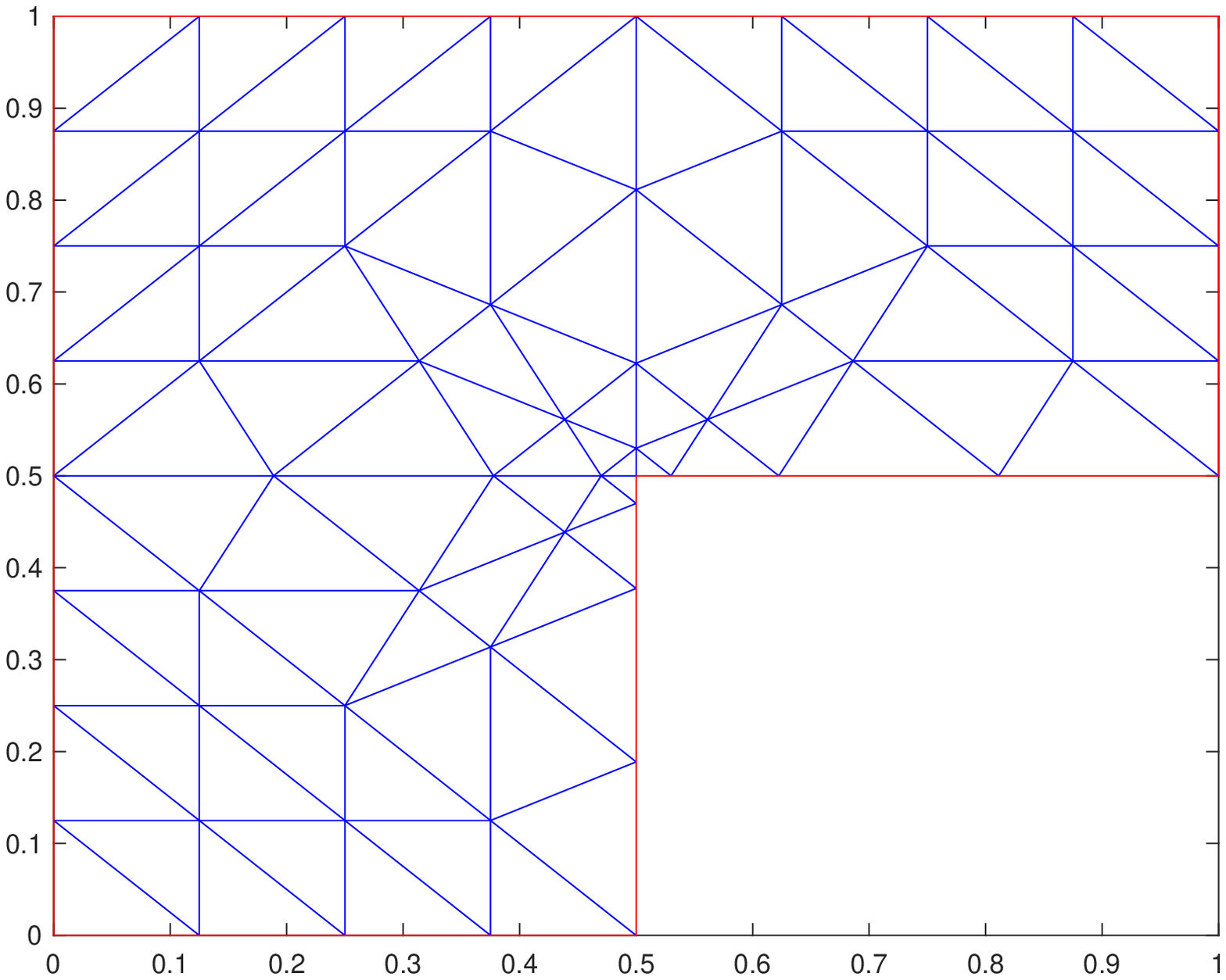}
	\includegraphics[width=0.32\linewidth, height=0.2\textheight]{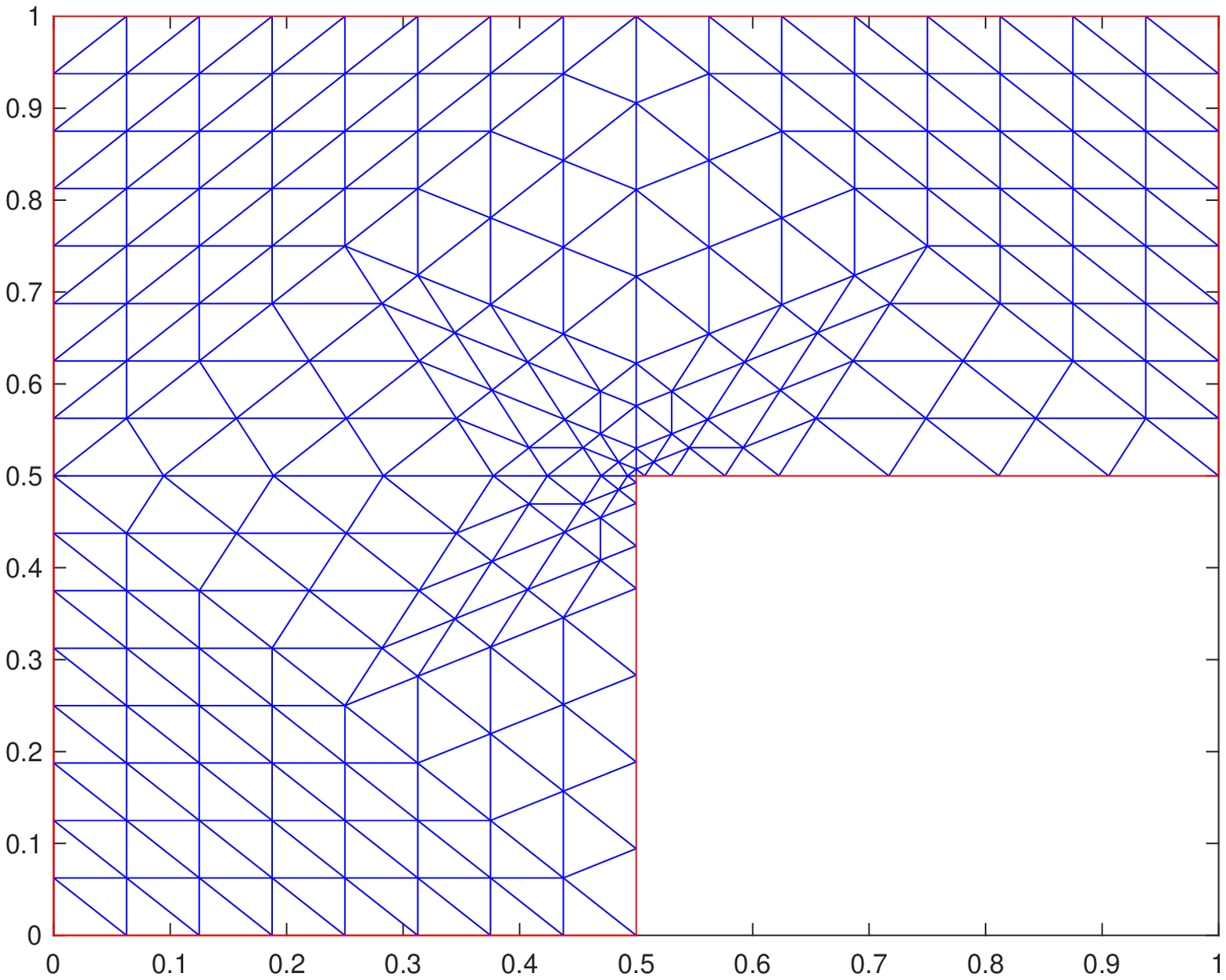}\\
	$n=1$\quad\quad\quad\quad\quad\quad\quad\quad\quad\quad\quad $n=2$\quad\quad\quad\quad\quad\quad\quad\quad\quad\quad\quad\quad\quad$n=3$
	\caption{graded mesh with $\kappa=\frac{|AB|}{|AC|}=0.245$}\label{kap245}
\end{figure}
\begin{table}[h]
	\centering
	\caption{Numerical results using the lowest-order $H^2(\tc)$ elements on triangles with $\kappa=0.5$} \label{tab4}
	\begin{tabular}{cccccccc}
		\hline
		$n$ &$\frac{\left\|\bm u_n-\bm u_{n+1}\right\|}{\left\|\bm u_{n+1}\right\|}$&order&$\frac{\left\|\nabla\times(\bm u_n-\bm u_{n+1})\right\|}{\left\|\nabla\times\bm u_{n+1}\right\|}$& order&$\frac{\left\|(\nabla\times)^2(\bm u_n-\bm u_{n+1})\right\|}{\left\|(\nabla\times)^2\bm u_{n+1}\right\|}$& order\\
		\hline
		$1$&8.1356190941e-03&1.7548&1.4488729983e-02&2.9705&6.5971068288e-02 &1.9539\\
		$2$&2.4107073619e-03 &1.3648 &1.8484963524e-03&2.7463&1.7028612632e-02 &1.7428\\
		$3$&9.3603277154e-04&1.3374& 2.7548691193e-04&2.0677&5.0881371807e-03&1.2316\\
		$4$&3.7041213929e-04&1.3345 &6.5715704328e-05&1.5270&2.1668185615e-03&0.8163\\
		$5$&1.4687454186e-04&&2.2803619189e-05&&1.2304934879e-03&\\
		\hline
	\end{tabular}
\end{table}

\begin{table}[h]
	\centering
	\caption{Numerical results using the lowest-order $H^2(\tc)$ elements on triangles with $\kappa=0.245$}\label{tab5}
	\begin{tabular}{cccccccc}
		\hline
		$n$ &$\frac{\left\|\bm u_n-\bm u_{n+1}\right\|}{\left\|\bm u_{n+1}\right\|}$&order&$\frac{\left\|\nabla\times(\bm u_n-\bm u_{n+1})\right\|}{\left\|\nabla\times\bm u_{n+1}\right\|}$& order&$\frac{\left\|(\nabla\times)^2(\bm u_n-\bm u_{n+1})\right\|}{\left\|(\nabla\times)^2\bm u_{n+1}\right\|}$& order\\
		\hline
		$1$&6.7412165834e-03&3.4251&1.6654125070e-02&2.9941&7.0276797462e-02&1.9870\\
		$2$&6.2761308517e-04&3.0702 &2.0903090109e-03&2.9976&1.7727617962e-02&1.9811\\
		$3$&7.4723734302e-05&2.8464& 2.6171918736e-04&2.9989&4.4901969525e-03&1.9596\\
		$4$&1.0390127924e-05&2.7527 &3.2740708974e-05&2.9989&1.1544428918e-03&1.9104\\
		$5$&1.5415683941e-06&&4.0957086477e-06&&3.0711331564e-04&\\
		\hline
	\end{tabular}
\end{table}

\section{Conclusion}
In this paper, we discuss $H^2$(curl)-conforming finite elements in 2-D and employ them to solve the quad-curl problem.  In details, we describe the elements and estimate the interpolation error. Furthermore, we obtain a optimal error estimate for the numerical scheme. The numerical experiments are provided to show the correctness of our theory and the efficiency of our scheme which is compared with the mixed finite element method in \cite{Sun2016A} and the Hodge decomposition method \cite{Brenner2017Hodge}. In the experiments, we find that there exist some superconvergence phenomena, which will be given a detailed proof  in our future work. Besides, we are also interested in the $H^2$(curl)-conforming finite elements in 3-D.


\section*{Appendix}
\begin{figure}
	\centering
	\includegraphics[width=0.4\linewidth, height=0.2\textheight]{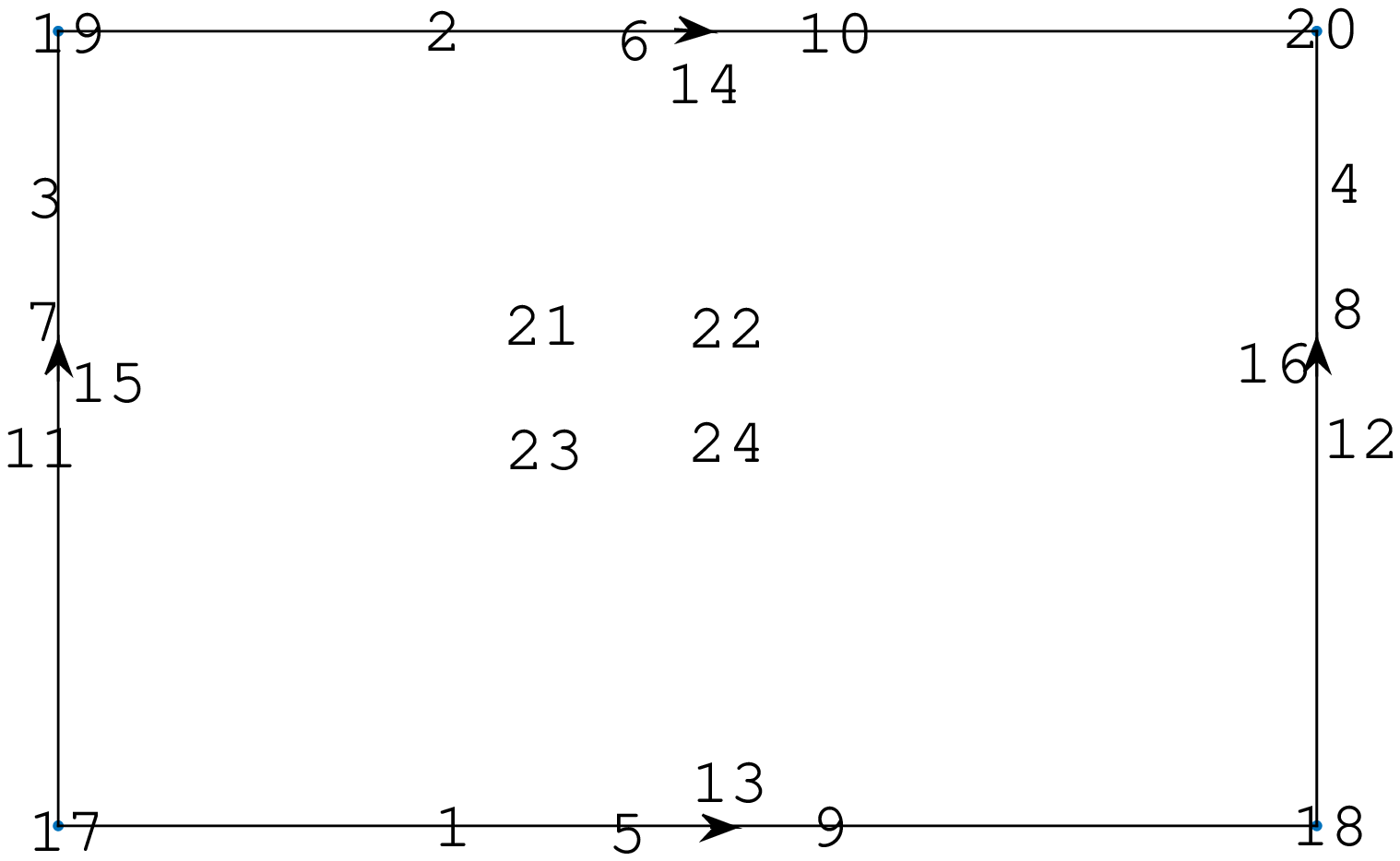}
	\includegraphics[width=0.4\linewidth, height=0.2\textheight]{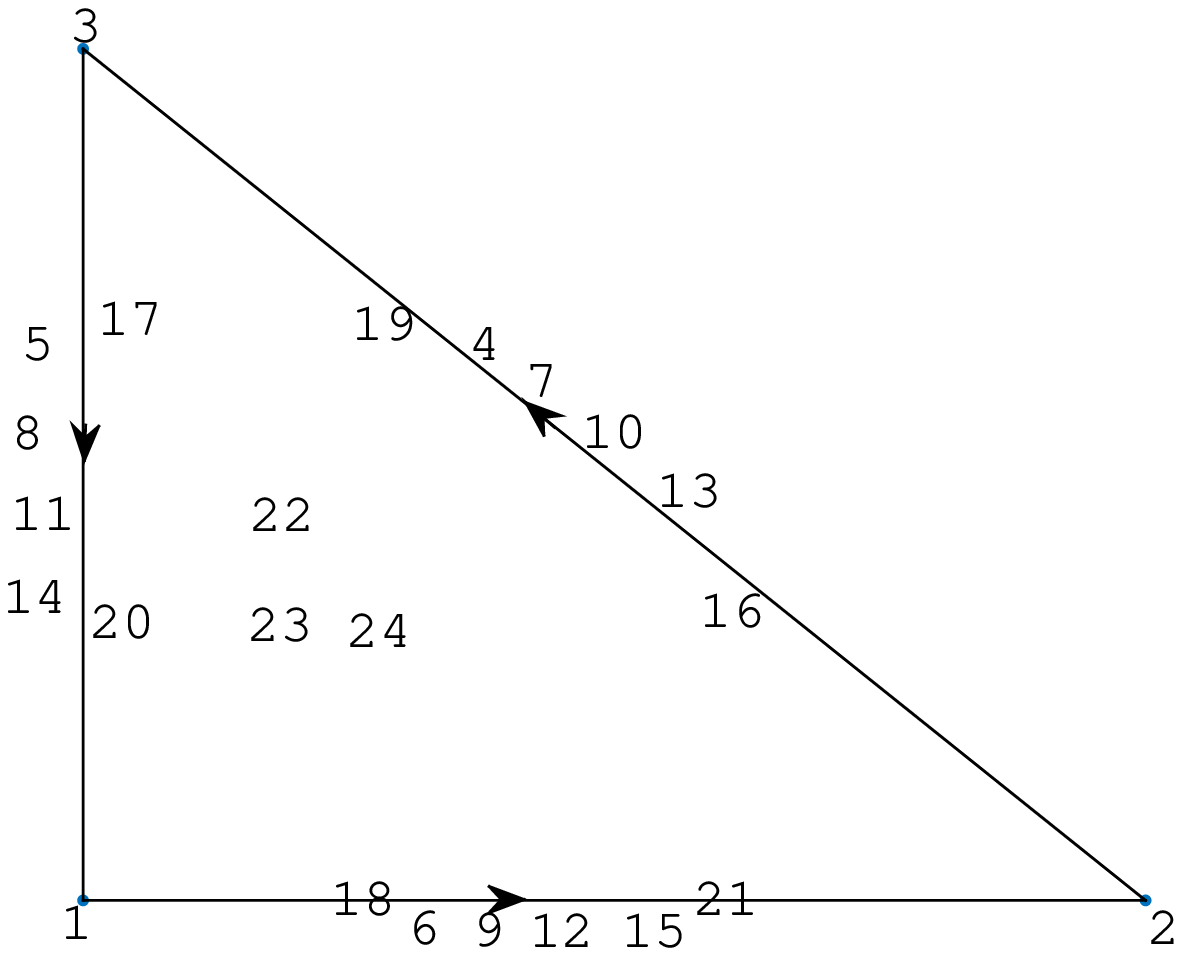}
	\caption{The sequence of DOFs. The arrows represent directions of tangent vectors. 13, 14, 15, 16, 17, 18, 19, 20 in the left and 1, 2, 3, 16, 17, 18, 19, 20, 21 in the right are the node DOFs. }
	\label{fig1}
\end{figure}

Now we list the basis functions of the  lowest-order $H^2$(curl)-conforming finite element. According to the interpolation conditions, the basis functions $\phi_i=(\phi^1_i,\phi^2_i)$ ($i=1,\;2,\cdots,\; 24$) in rectangular $\hat{K}$ can be founded analytically as follows. Note that the order of DOFs is shown in Figure \ref{fig1}.
\begin{align*}
\phi^1_1&=207 x^2 y/128 - 87 x^2 y^3 /128 - 15 x^2 /16 + 45 y^3 /128 - 117 y /128 + 9/16 ,\\
\phi^1_2&=87 x^2 y^3 /128 - 207 x^2 y /128 - 15 x^2 /16 - 45 y^3 /128 + 117 y /128 + 9/16 ,\\
\phi^1_3&=63 x^2 y /128 - 63 x^2 y^3 /128 + 45 y^3 /128 - 45 y /128 ,\\
\phi^1_4&=63 x^2 y^3 /128 - 63 x^2 y /128 - 45 y^3 /128 + 45 y /128 ,\\
\phi^1_5&=9 x y^2 /16 - 3 x y /4 + 3 x /16 ,\\
\phi^1_6&=9 x y^2 /16 + 3 x y /4 + 3 x /16 ,\\
\phi^1_7&=9 x y^2 /16 - 9 x /16 - 3 y^2 /8 + 3/8 ,\\
\phi^1_8&=9 x y^2 /16 - 9 x /16 + 3 y^2 /8 - 3/8 ,\\
\phi^1_9&=45 x^2 y^2 /16 - 225 x^2 y^3 /128 - 135 x^2 y /128 + 75 y^3 /128 - 15 y^2 /16 + 45 y /128 ,\\
\phi^1_{10}&=225 x^2 y^3 /128 + 45 x^2 y^2 /16 + 135 x^2 y /128 - 75 y^3 /128 - 15 y^2 /16 - 45 y /128 ,\\
\phi^1_{11}&=225 x^2 y /128 - 225 x^2 y^3 /128 + 15 x y^3 /8 - 15 x y /8 - 45 y^3 /128 + 45 y /128 ,\\
\phi^1_{12}&=225 x^2 y^3 /128 - 225 x^2 y /128 + 15 x y^3 /8 - 15 x y /8 + 45 y^3 /128 - 45 y /128 ,\\
\phi^1_{13}&=7 x^2 y^3 /16 - x^2 y^2 /2 - 7 x^2 y /16 + x^2/2 - 5 y^3 /16 + y^2/3 + 5 y /16 - 1/3 ,\\
\phi^1_{14}&=7 x^2 y^3 /16 + x^2 y^2 /2 - 7 x^2 y /16 - x^2/2 - 5 y^3 /16 - y^2/3 + 5 y /16 + 1/3 ,\\
\phi^1_{15}&=3 x^2 y /16 - 3 x^2 y^3 /16 + x y^3 /6 - x y /6 + y^3/16 - y/16 ,\\
\phi^1_{16}&=3 x^2 y /16 - 3 x^2 y^3 /16 - x y^3 /6 + x y /6 + y^3/16 - y/16 ,\\
\phi^1_{17}&=(- x^2 y^3 - 2 x^2 y^2 + x^2 y + 2 x^2 + 4 x y^3 - 2 x y^2 - 4 x y + 2 x - y^3 + 2 y^2 + y - 2)/32,\\
\phi^1_{18}&=(- x^2 y^3 - 2 x^2 y^2 + x^2 y + 2 x^2 - 4 x y^3 + 2 x y^2 + 4 x y - 2 x - y^3 + 2 y^2 + y - 2)/32,\\
\phi^1_{19}&=(- x^2 y^3 + 2 x^2 y^2 + x^2 y - 2 x^2 + 4 x y^3 + 2 x y^2 - 4 x y - 2 x - y^3 - 2 y^2 + y + 2)/32,\\
\phi^1_{20}&=(- x^2 y^3 + 2 x^2 y^2 + x^2 y - 2 x^2 - 4 x y^3 - 2 x y^2 + 4 x y + 2 x - y^3 - 2 y^2 + y + 2)/32,\\
\phi^1_{21}&=9 x /16 - 9 x y^2 /16 ,\\
\phi^1_{22}&=45 x^2 /16 - 45 x^2 y^2 /16 + 15 y^2 /16 - 15/16 ,\\
\phi^1_{23}&=15 x y /8 - 15 x y^3 /8 ,\\
\phi^1_{24}&=675 x^2 y /64 - 675 x^2 y^3 /64 + 225 y^3 /64 - 225 y /64.
\end{align*}
\begin{align*}
\phi^2_{1}&=45 x^3 /128 -  63 x^3 y^2 /128 +  63 x y^2 /128 -  45 x /128 ,\\
\phi^2_{2}&=63 x^3 y^2 /128 -  45 x^3 /128 -  63 x y^2 /128 +  45 x /128 ,\\
\phi^2_{3}&=45 x^3 /128 -  87 x^3 y^2 /128 +  207 x y^2 /128 -  117 x /128 -  15 y^2 /16 + 9/16 ,\\
\phi^2_{4}&= 87 x^3 y^2 /128 -  45 x^3 /128 -  207 x y^2 /128 +  117 x /128 -  15 y^2 /16 + 9/16 ,\\
\phi^2_{5}&=9 x^2 y /16 -  9 y /16 -  3 x^2 /8 + 3/8 ,\\
\phi^2_{6}&=9 x^2 y /16 -  9 y /16 +  3 x^2 /8 - 3/8 ,\\
\phi^2_{7}&=9 y x^2 /16 -  3 y x /4 +  3 y /16 ,\\
\phi^2_{8}&=9 y x^2 /16 +  3 y x /4 +  3 y /16 ,\\
\phi^2_{9}&=15 x^3 y /8 -  225 x^3 y^2 /128 -  45 x^3 /128 +  225 x y^2 /128 -  15 x y /8 +  45 x /128 ,\\
\phi^2_{10}&=225 x^3 y^2 /128 +  15 x^3 y /8 +  45 x^3 /128 -  225 x y^2 /128 -  15 x y /8 -  45 x /128 ,\\
\phi^2_{11}&=75 x^3 /128 -  225 x^3 y^2 /128 +  45 x^2 y^2 /16 -  15 x^2 /16 -  135 x y^2 /128 +  45 x /128 ,\\
\phi^2_{12}&=225 x^3 y^2 /128 -  75 x^3 /128 +  45 x^2 y^2 /16 -  15 x^2 /16 +  135 x y^2 /128 -  45 x /128 ,\\
\phi^2_{13}&=3 x^3 y^2 /16 -  x^3 y /6 - x^3/16 -  3 x y^2 /16 +  x y /6 + x/16 ,\\
\phi^2_{14}&=3 x^3 y^2 /16 +  x^3 y /6 - x^3/16 -  3 x y^2 /16 -  x y /6 + x/16 ,\\
\phi^2_{15}&=5 x^3 /16 -  7 x^3 y^2 /16 +  x^2 y^2 /2 - x^2/3 +  7 x y^2 /16 -  5 x /16 - y^2/2 + 1/3 ,\\
\phi^2_{16}&=5 x^3 /16 -  7 x^3 y^2 /16 -  x^2 y^2 /2 + x^2/3 +  7 x y^2 /16 -  5 x /16 + y^2/2 - 1/3 ,\\
\phi^2_{17}&=(x^3 y^2 - 4 x^3 y + x^3 + 2 x^2 y^2 + 2 x^2 y - 2 x^2 - x y^2 + 4 x y - x - 2 y^2 - 2 y + 2)/32,\\
\phi^2_{18}&=(x^3 y^2 - 4 x^3 y + x^3 - 2 x^2 y^2 - 2 x^2 y + 2 x^2 - x y^2 + 4 x y - x + 2 y^2 + 2 y - 2)/32,\\
\phi^2_{19}&=(x^3 y^2 + 4 x^3 y + x^3 + 2 x^2 y^2 - 2 x^2 y - 2 x^2 - x y^2 - 4 x y - x - 2 y^2 + 2 y + 2)/32,\\
\phi^2_{20}&=(x^3 y^2 + 4 x^3 y + x^3 - 2 x^2 y^2 + 2 x^2 y + 2 x^2 - x y^2 - 4 x y - x + 2 y^2 - 2 y - 2)/32,\\
\phi^2_{21}&=9 y /16 -  9 x^2 y /16 ,\\
\phi^2_{22}&=15 x y /8 -  15 x^3 y /8 ,\\
\phi^2_{23}&=15 x^2 /16 -  45 x^2 y^2 /16 +  45 y^2 /16 - 15/16 ,\\
\phi^2_{24}&=225 x^3 /64 -  675 x^3 y^2 /64 +  675 x y^2 /64 -  225 x /64.
\end{align*}

Similarly, we can obtain the standard basis functions in the triangle $\hat{K}$ which is denoted as $\phi_i=(\phi^1_i,\phi^2_i)$ ($i=1,\;2,\cdots,\; 24$).  The order of DOFs is shown in Figure \ref{fig1}.
\begin{align*}
\phi^1_{1}=&   9x^3y /10 +  23x^2y^2 /10 -  5x^2y /2 +  23xy^3 /10 - 4xy^2  + 59xy /30 +  9y^4 /10 -  13y^3 /6 \\
&+  53y^2 /30 - y/2 ,\\
\phi^1_{2}=&   71x^2y /55 -  2x^2y^2 /5 -  9x^3y /10 -  2xy^3 /5 +  83xy^2 /110 -  2191xy /3300 -  101y^3 /660 \\
&+  333y^2 /2200 + y/600 ,\\
\phi^1_{3}=&   101x^2y /220 -  2x^2y^2 /5 -  2xy^3 /5 +  27xy^2 /110 -  333xy /1100 -  9y^4 /10 +  353y^3 /330 \\
&-  1109y^2 /6600 - y/600 ,\\
\phi^1_{4}=&  24x^2y^2 + 3x^2y + 24xy^3 - 30xy^2 - 5xy - y^3 +  5y^2 /2 ,\\
\phi^1_{5}=&  24x^2y^2 + 3x^2y + 24xy^3 - 30xy^2 - 5xy - y^3 +  5y^2 /2 ,\\
\phi^1_{6}=&  24x^2y^2 + 3x^2y - 15x^2 + 24xy^3 - 30xy^2 - 5xy + 15x - y^3 +  5y^2 /2 - 3/2 ,\\
\phi^1_{7}=&   1800x^2y /11 +  90xy^2 /11 -  1860xy /11 +  600y^3 /11 -  930y^2 /11 + 45y ,\\
\phi^1_{8}=&   -1260x^2y /11 +  630xy^2 /11 + 708xy /11 - 420y^3 /11 + 354y^2 /11 - 9y ,\\
\phi^1_{9}=&  630x^2 -  1260x^2y /11 - 420x^3 +  630xy^2 /11 +  708xy /11 - 240x -  420y^3 /11 +  354y^2 /11 \\
&- 9y + 15 ,\\
\phi^1_{10}=&  60xy - 30y^2 ,\\
\phi^1_{11}=&   1440xy^2 /11 -  900x^2y /11 -  60xy /11 +  1020y^3 /11 -  2010y^2 /11 + 60y ,\\
\phi^1_{12}=&( - 3060 x^2 y + 1980 x^2 - 1440 x y^2 + 4020 x y - 1980 x + 300 y^3 + 30 y^2 - 660 y + 330)/11,\\
\phi^1_{13}=&   4200xy^2 /11 -  8400x^2y /11 +  5600xy /11 -  2800y^3 /11 +  2800y^2 /11 - 140y ,\\
\phi^1_{14}=&   -4200xy^2 /11 +  8400x^2y /11 - 5600xy /11 +2800y^3 /11 -  2800y^2 /11 + 140y ,\\
\phi^1_{15}=&  2800x^3 +  8400x^2y /11 - 4200x^2 -  4200xy^2 /11 -  5600xy /11 + 1680x +  2800y^3 /11\\
& -  2800y^2 /11 + 140y - 140 ,\\
\phi^1_{16}=&( - 1584 x^2 y^2 + 1062 x^2 y - 396 x y^3 + 1152 x y^2 - 906 x y + 24 y^3 - 123 y^2 + 99 y)/440,\\
\phi^1_{17}=&   9xy^3 /5 -  81x^2y /110 -  9x^2y^2 /10 -  54xy^2 /55 +  567xy /550 +  27y^4 /10 -  261y^3 /55 \\
&+  2547y^2 /1100 -  27y /100 ,\\
\phi^1_{18}=&   387x^2y /55 -  63x^2y^2 /10 -  27x^3y /10 -  18xy^3 /5 +  351xy^2 /55 -  1953xy /550 -  36y^3 /55\\
& +  1017y^2 /1100 -  27y /100 ,\\
\phi^1_{19}=& (- 396 x^2 y^2 - 72 x^2 y - 1584 x y^3 + 828 x y^2 + 246 x y - 354 y^3 + 453 y^2 - 99 y)/440,\\
\phi^1_{20}=&   108x^2y /55 -  18x^2y^2 /5 -  63xy^3 /10 +  783xy^2 /110 -  1017xy /550 -  27y^4 /10 +  567y^3 /110\\
& -  2997y^2 /1100 +  27y /100 ,\\
\phi^1_{21}=&   27x^3y /10 +  9x^2y^2 /5 -  207x^2y /55 -  9xy^3 /10 +  54xy^2 /55 -  36xy /275 +  27y^3 /110 \\
&-  567y^2 /1100 +  27y /100 ,\\
\phi^1_{22}=&   1152xy /11 -  720xy^2 /11 -  540x^2y /11 -  180y^3 /11 +  576y^2 /11 - 36y ,\\
\phi^1_{23}=&   3060x^2y /11 +  1440xy^2 /11 -  3360xy /11 -  300y^3 /11 -  360y^2 /11 + 60y ,\\
\phi^1_{24}=&   1440xy^2 /11 -  900x^2y /11 -  720xy /11 +  1020y^3 /11 -  1680y^2 /11 + 60y .
\end{align*}

\begin{align*}
\phi^2_{1}=&  13x^3 /6 -  23x^3y /10 -  9x^4 /10 -  23x^2y^2 /10 + 4x^2y -  53x^2 /30 -  9xy^3 /10 +  5xy^2 /2\\
& -  59xy /30 + x/2 ,\\
\phi^2_{2}=&  9x^4 /10 +  2x^3y /5 -  353x^3 /330 +  2x^2y^2 /5 -  27x^2y /110 +  1109x^2 /6600 -  101xy^2 /220\\
& +  333xy /1100 + x/600 ,\\
\phi^2_{3}=&  2x^3y /5 +  101x^3 /660 +  2x^2y^2 /5 -  83x^2y /110 -  333x^2 /2200 +  9xy^3 /10 -  71xy^2 /55 \\
&+  2191xy /3300 - x/600 ,\\
\phi^2_{4}=& x^3 - 24x^3y - 24x^2y^2 + 30x^2y -  5x^2 /2 - 3xy^2 + 5xy ,\\
\phi^2_{5}=& x^3 - 24x^3y - 24x^2y^2 + 30x^2y -  5x^2 /2 - 3xy^2 + 5xy + 15y^2 - 15y + 3/2 ,\\
\phi^2_{6}=& x^3 - 24x^3y - 24x^2y^2 + 30x^2y -  5x^2 /2 - 3xy^2 + 5xy ,\\
\phi^2_{7}=&  600x^3 /11 +  90x^2y /11 -  930x^2 /11 +  1800xy^2 /11 -  1860xy /11 + 45x ,\\
\phi^2_{8}=& - 420x^3 /11 +  630x^2y /11 +  354x^2 /11 -  1260xy^2 /11 + 708xy /11 - 9x - 420y^3 +630y^2\\
& -240y +15 ,\\
\phi^2_{9}=&  630x^2y /11 -  420x^3 /11 +  354x^2 /11 -  1260xy^2 /11 +  708xy /11 - 9x ,\\
\phi^2_{10}=& 30x^2 - 60yx ,\\
\phi^2_{11}=&(- 300 x^3 + 1440 x^2 y - 30 x^2 + 3060 x y^2 - 4020 x y + 660 x - 1980 y^2 + 1980 y - 330)/11 ,\\
\phi^2_{12}=&  2010x^2 /11 -  1440x^2y /11 -  1020x^3 /11 +  900xy^2 /11 +  60xy /11 - 60x ,\\
\phi^2_{13}=&  4200x^2y /11 -  2800x^3 /11 +  2800x^2 /11 -  8400xy^2 /11 +  5600xy /11 - 140x ,\\
\phi^2_{14}=&  -4200x^2y /11 +  2800x^3 /11 -  2800x^2 /11 +  8400xy^2 /11 - 5600xy /11 + 140x + 2800y^3\\
&  - 4200y^2+ 1680y - 140 ,\\
\phi^2_{15}=&  2800x^3 /11 -  4200x^2y /11 -  2800x^2 /11 +  8400xy^2 /11 -  5600xy /11 + 140x ,\\
\phi^2_{16}=&(1584 x^3 y + 354 x^3 + 396 x^2 y^2 - 828 x^2 y - 453 x^2 + 72 x y^2 - 246 x y + 99 x)/440 ,\\
\phi^2_{17}=&  9x^3y /10 -  27x^3 /110 -  9x^2y^2 /5 -  54x^2y /55 +  567x^2 /1100 -  27xy^3 /10 +  207xy^2 /55 \\
&+  36xy /275 -  27x /100 ,\\
\phi^2_{18}=&  27x^4 /10 +  63x^3y /10 -  567x^3 /110 +  18x^2y^2 /5 -  783x^2y /110 +  2997x^2 /1100 -  108xy^2 /55\\
& +  1017xy /550 -  27x /100 ,\\
\phi^2_{19}=&  (396 x^3 y - 24 x^3 + 1584 x^2 y^2 - 1152 x^2 y + 123 x^2 - 1062 x y^2 + 906 x y - 99 x)/440 ,\\
\phi^2_{20}=&  18x^3y /5 +  36x^3 /55 +  63x^2y^2 /10 -  351x^2y /55 -  1017x^2 /1100 +  27xy^3 /10 -  387xy^2 /55 \\
&+  1953xy /550 +  27x /100 ,\\
\phi^2_{21}=&  261x^3 /55 -  9x^3y /5 -  27x^4 /10 +  9x^2y^2 /10 +  54x^2y /55 -  2547x^2 /1100 +  81xy^2 /110\\
& -  567xy /550 +  27x /100 ,\\
\phi^2_{22}=&  576x^2 /11 -  720x^2y /11 -  180x^3 /11 -  540xy^2 /11 +  1152xy /11 - 36x ,\\
\phi^2_{23}=&  1020x^3 /11 +  1440x^2y /11 -  1680x^2 /11 -  900xy^2 /11 -  720xy /11 + 60x ,\\
\phi^2_{24}=&  1440x^2y /11 -  300x^3 /11 -  360x^2 /11 +  3060xy^2 /11 -  3360xy /11 + 60x .
\end{align*}

\end{document}